\documentclass[12pt,leqno]{amsart}
\usepackage{latexsym}
\usepackage{amsmath}
\usepackage{amsthm}
\usepackage{amssymb}
\usepackage{amsfonts}

\usepackage{comment}
\begin{document}
\theoremstyle{plain}
\newtheorem{MainThm}{Theorem}
\newtheorem{thm}{Theorem}[section]
\newtheorem{clry}[thm]{Corollary}
\newtheorem{prop}[thm]{Proposition}
\newtheorem{lemma}[thm]{Lemma}
\newtheorem{deft}[thm]{Definition}
\newtheorem{hyp}{Assumption}
\newtheorem*{conjecture}{Conjecture}
\newtheorem*{question}{Question}

\newtheorem{claim}[thm]{Claim}

\theoremstyle{definition}
\newtheorem*{definition}{Definition}
\newtheorem{assumption}{Assumption}
\newtheorem{rem}[thm]{Remark}
\newtheorem*{remark}{Remark}
\newtheorem*{acknow}{Acknowledgements}

\newtheorem{example}[thm]{Example}
\newtheorem*{examplenonum}{Example}
\numberwithin{equation}{section}
\newcommand{\nocontentsline}[3]{}
\newcommand{\tocless}[2]{\bgroup\let\addcontentsline=\nocontentsline#1{#2}\egroup}
\newcommand{\eps}{\varepsilon}
\renewcommand{\phi}{\varphi}
\renewcommand{\d}{\partial}
\newcommand{\re}{\mathop{\rm Re} }
\newcommand{\im}{\mathop{\rm Im}}
\newcommand{\mR}{\mathbb{R}}
\newcommand{\mC}{\mathbb{C}}
\newcommand{\mN}{\mathbb{N}} 
\newcommand{\mZ}{\mathbb{Z}} 
\newcommand{\mK}{\mathbb{K}}
\newcommand{\supp}{\mathop{\rm supp}}
\newcommand{\abs}[1]{\lvert #1 \rvert}
\newcommand{\norm}[1]{\lVert #1 \rVert}
\newcommand{\csubset}{\Subset}
\newcommand{\detg}{\lvert g \rvert}
\newcommand{\msetminus}{\setminus}

\newcommand{\br}[1]{\langle #1 \rangle}

\newcommand{\ehat}{\,\hat{\rule{0pt}{6pt}}\,}
\newcommand{\echeck}{\,\check{\rule{0pt}{6pt}}\,}
\newcommand{\etilde}{\,\tilde{\rule{0pt}{6pt}}\,}

\newcommand{\tr}{\mathrm{tr}}
\newcommand{\mdiv}{\mathrm{div}}

\title[Calder\'on problem with partial data]{The Calder\'on problem with partial data on manifolds and applications}

\author{Carlos Kenig}
\address{Department of Mathematics, University of Chicago}
\email{cek@math.uchicago.edu}

\author{Mikko Salo}
\address{Department of Mathematics and Statistics, University of Jyv\"askyl\"a}
\email{mikko.j.salo@jyu.fi}


\begin{abstract}
We consider Calder\'on's inverse problem with partial data in dimensions $n \geq 3$. If the inaccessible part of the boundary satisfies a (conformal) flatness condition in one direction, we show that this problem reduces to the invertibility of a broken geodesic ray transform. In Euclidean space, sets satisfying the flatness condition include parts of cylindrical sets, conical sets, and surfaces of revolution. We prove local uniqueness in the Calder\'on problem with partial data in admissible geometries, and global uniqueness under an additional concavity assumption. This work unifies two earlier approaches to this problem (\cite{KSU} and \cite{I}) and extends both. The proofs are based on improved Carleman estimates with boundary terms, complex geometrical optics solutions involving reflected Gaussian beam quasimodes, and invertibility of (broken) geodesic ray transforms. This last topic raises questions of independent interest in integral geometry.
\end{abstract}

\maketitle

\tableofcontents

\section{Introduction} \label{sec_intro}

This article is concerned with inverse problems where measurements are made only on part of the boundary. A typical example is the inverse problem of Calder\'on, where the objective is to determine the electrical conductivity of a medium from voltage and current measurements on its boundary. The mathematical formulation of this problem is as follows. Let $\Omega \subset \mR^n$, $n \geq 2$, be a bounded domain with smooth boundary. Given a positive function $\gamma \in L^{\infty}(\Omega)$ (the electrical conductivity of the medium) and two open subsets $\Gamma_D, \Gamma_N$ of $\partial \Omega$, consider the partial Cauchy data set 
\begin{align*}
C_{\gamma}^{\Gamma_D,\Gamma_N} = \{ (u|_{\Gamma_D}, \gamma \partial_{\nu} u|_{\Gamma_N}) &\,;\, \mdiv(\gamma \nabla u) = 0 \text{ in } \Omega, \ u \in H^1(\Omega), \\
 & \quad \supp(u|_{\partial \Omega}) \subset \Gamma_D \}.
\end{align*}
The Calder\'on problem with partial data asks to determine the conductivity $\gamma$ from the knowledge of $C_{\gamma}^{\Gamma_D, \Gamma_N}$ for possibly very small sets $\Gamma_D, \Gamma_N$. Here $\partial_{\nu}$ is the normal derivative, and the conormal derivative $\gamma \partial_{\nu} u|_{\partial \Omega}$ is interpreted in the weak sense as an element of $H^{-1/2}(\partial \Omega)$. 

A closely related problem is to determine a potential $q \in L^{\infty}(\Omega)$ from partial boundary measurements for the Schr\"odinger equation, given by the partial Cauchy data set 
\begin{align*}
C_{q}^{\Gamma_D,\Gamma_N} = \{ (u|_{\Gamma_D}, \partial_{\nu} u|_{\Gamma_N}) &\,;\, (-\Delta+q) = 0 \text{ in } \Omega, \ u \in H_{\Delta}(\Omega), \\
 & \quad \supp(u|_{\partial \Omega}) \subset \Gamma_D \}.
\end{align*}
Here we use the space 
$$
H_{\Delta}(\Omega) = \{ u \in L^2(\Omega) \,;\, \Delta u \in L^2(\Omega) \},
$$
and the trace $u|_{\partial \Omega}$ and normal derivative $\partial_{\nu} u|_{\partial \Omega}$ are in $H^{-1/2}(\partial \Omega)$ and $H^{-3/2}(\partial \Omega)$ (see \cite{BU}). Above, one thinks of $u|_{\partial \Omega}$ as Dirichlet data prescribed only on $\Gamma_D$, and one measures the Neumann data of the corresponding solution on $\Gamma_N$. If $\Lambda_{\gamma}: H^{1/2}(\Omega) \to H^{-1/2}(\partial \Omega)$ is the Dirichlet-to-Neumann map (DN map) given by 
$$
\Lambda_{\gamma}: u|_{\partial \Omega} \mapsto \gamma \partial_{\nu} u|_{\partial \Omega} \ \ \text{where $u \in H^1(\Omega)$ solves $\mdiv(\gamma \nabla u) = 0$ in $\Omega$},
$$
then the partial Cauchy data set is a restriction of the graph of $\Lambda_{\gamma}$, 
$$
C_{\gamma}^{\Gamma_D,\Gamma_N} = \{ (f|_{\Gamma_D}, \Lambda_{\gamma} f|_{\Gamma_N}) \,;\, f \in H^{1/2}(\partial \Omega), \ \supp(f) \subset \Gamma_D \}.
$$
A similar interpretation is valid for $C_{q}^{\Gamma_D,\Gamma_N}$ provided that $0$ is not a Dirichlet eigenvalue of $-\Delta+q$ in $\Omega$.

The problems above are well studied questions in the theory of inverse problems. The case of full data ($\Gamma_D = \Gamma_N = \partial \Omega$) has received the most attention. Major results include \cite{SU}, \cite{HT} in dimensions $n \geq 3$ and \cite{N_2D}, \cite{AP}, \cite{Bu} in the case $n=2$. In particular, it is known that the set $C_{\gamma}^{\partial \Omega, \partial \Omega}$ determines uniquely a conductivity $\gamma \in C^1(\overline{\Omega})$ if $n \geq 3$ and a conductivity $\gamma \in L^{\infty}(\Omega)$ if $n = 2$. These results are based on the method of complex geometrical optics solutions developed in \cite{SU} for $n \geq 3$ and in \cite{N_2D}, \cite{Bu} in the case $n=2$.

The partial data question where the sets $\Gamma_D$ or $\Gamma_N$ may not be the whole boundary has also attracted considerable attention. We mention here four approaches, each of which gives a slightly different partial data result. Formulated in terms of the Schr\"odinger problem, it is known that $C_q^{\Gamma_D,\Gamma_N}$ determines $q$ in $\Omega$ in the following cases:
\begin{enumerate}
\item
$n \geq 3$, the set $\Gamma_D$ is possibly very small, and $\Gamma_N$ is slightly larger than $\partial \Omega \setminus \overline{\Gamma}_D$ (Kenig, Sj\"ostrand, and Uhlmann \cite{KSU}) \\
\item
$n \geq 3$ and $\Gamma_D = \Gamma_N = \Gamma$, and $\partial \Omega \setminus \Gamma$ is either part of a hyperplane or part of a sphere (Isakov \cite{I}) \\
\item
$n = 2$ and $\Gamma_D = \Gamma_N = \Gamma$, where $\Gamma$ can be an arbitrary open subset of $\partial \Omega$ (Imanuvilov, Uhlmann, and Yamamoto \cite{IUY}) \\
\item
$n \geq 2$, linearized partial data problem, $\Gamma_D = \Gamma_N = \Gamma$ where $\Gamma$ can be an arbitrary open subset of $\partial \Omega$ (Dos Santos, Kenig, Sj\"ostrand, and Uhlmann \cite{DKSjU09})
\end{enumerate}

Approaches (1)--(3) also give a partial data result of determining $\gamma$ from $C_{\gamma}^{\Gamma_D,\Gamma_N}$ with the same assumptions on the dimension and the sets $\Gamma_D, \Gamma_N$. In (4), the linearized partial data problem asks to show injectivity of the Fr\'echet derivative of $\Lambda_q$ at $q=0$ instead of injectivity of the full map $q \mapsto \Lambda_q$, when restricted to the sets $\Gamma_D$ and $\Gamma_N$.

It is interesting that although each of the four approaches is based on a version of complex geometrical optics solutions, the approaches are distinct in the sense that none of the above results is contained in any of the others. The result in \cite{KSU} uses Carleman estimates with boundary terms, given for special limiting weights, that allow to control the solutions on parts of the boundary, whereas \cite{I} is based on the full data arguments of \cite{SU} and a reflection argument. The work \cite{IUY} gives a strong result that only requires Dirichlet and Neumann data on any small set, but the method involves complex analysis and Carleman weights with critical points and does not obviously extend to higher dimensions. Finally, \cite{DKSjU09} is based on analytic microlocal analysis but is so far restricted to the linearized problem.

Nevertheless, given that there exist several approaches to the same problem, one expects that a combination of ideas from different approaches might lead to improved partial data results. In this paper we unify the Carleman estimate approach of \cite{KSU} and the reflection approach of \cite{I}, and in fact we obtain the main results of both \cite{KSU} and \cite{I} as a special case.

The method also allows to improve both approaches. Concerning \cite{I}, we are able to relax the hypothesis on the inaccessible part $\Gamma_i = \partial \Omega \setminus \Gamma$ of the boundary: instead of requiring $\Gamma_i$ to be completely flat (or spherical), we can deal with $\Gamma_i$ that satisfy a flatness condition only in one direction. Compared with \cite{KSU} we remove the need of measurements on certain parts of the boundary that are flat in one direction; and in certain cases where $\partial \Omega$ may not have any symmetries, we eliminate the overlap of $\Gamma_D$ and $\Gamma_N$ needed in \cite{KSU}. The method eventually boils down to inverting geodesic ray transforms (possibly for broken geodesics). In some cases the invertibility of the ray transform is known, but in other cases it is not and in these cases we obtain a reduction from the Calder\'on problem with partial data to integral geometry problems of independent interest.

We list here some further references for partial data results, first for the case $n \geq 3$. The Carleman estimate approach was initiated in \cite{BU} and \cite{KSU}. Based on this approach, there are low regularity results \cite{Knudsen}, \cite{Zhang}, results for other scalar equations \cite{DKSjU07}, \cite{KnudsenSalo}, \cite{Chung} and systems \cite{SaloTzou}, stability results \cite{HeckWang}, and reconstruction results \cite{NS}. The reflection approach was introduced in \cite{I}, and has been employed for the Maxwell system \cite{COS}. Partial data results for slab geometries are given in \cite{LiUhlmann}, \cite{KLU}. Also, just before this preprint was submitted, Imanuvilov and Yamamoto announced a partial data result for domains $Q = \Omega \times (0,l)$ stating that the Cauchy data set of a potential $q \in C^{\alpha}(Q)$ with Dirichlet and Neumann data restricted to $\Gamma \times (0,l) $, where $\Gamma \subset \partial \Omega$ is an open subset, determines the potential in the set $(\Omega \setminus \Omega_0) \times (0,l)$ where $\Omega_0$ is the convex hull of $\partial \Omega \setminus \Gamma$. This is similar to the results in Section \ref{subsec_cylindrical} in this paper.

In two dimensions, the main partial data result is \cite{IUY} which has been extended to more general equations \cite{IUY_general}, combinations of measurements on disjoint sets \cite{IUY_disjoint}, less regular coefficients \cite{IY_nonsmooth}, and some systems \cite{IY_systems}. An earlier result is in \cite{ALP}. In the case of Riemann surfaces with boundary, corresponding partial data results are given in \cite{GT}, \cite{GT_magnetic}, \cite{GT_system}.

In the case when the conductivity is known near the boundary, the partial data problem can be reduced to the full data problem  \cite{AU}, \cite{Alessandrini_partial}, \cite{HPS}. Also, we remark that in the corresponding problem for the wave equation, it has been known for a long time (see \cite{KKL}) that measuring the Dirichlet and Neumann data of waves on an arbitrary open subset of the boundary is sufficient to determine the coefficients uniquely up to natural gauge transforms. Recent partial results for the case where Dirichlet and Neumann data are measured on disjoint sets are in \cite{LO1}, \cite{LO2}.

The structure of this paper is as follows. Section \ref{sec_intro} is the introduction. Section \ref{sec_results} states the main partial data results in this paper in the setting of Riemannian manifolds, and Section \ref{sec_euclidean} considers some consequences for the Calder\'on problem with partial data in Euclidean space. Section \ref{sec_carleman} gives a Carleman estimate that is used to control solutions on parts of the boundary, and Section \ref{sec_reflection} discusses a reflection approach that can be used as an alternative to Carleman estimates in some cases. In Section \ref{sec_simple} we give the proofs of the local uniqueness results for simple transversal manifolds, based on complex geometrical optics solutions involving WKB type quasimodes. In Section \ref{sec_gaussian} we discuss a more sophisticated quasimode construction based on reflected Gaussian beams, and in Section \ref{sec_brokenraytransform} we show how complex geometrical optics solutions involving reflected Gaussian beam quasimodes can be used to recover the broken ray transform of a potential from partial Cauchy data.


\bigskip

\noindent {\bf Acknowledgements.} \ 
C.K.~is partly supported by NSF, and M.S.~is supported in part by the Academy of Finland and an ERC Starting Grant. M.S.~would like to thank David Dos Santos Ferreira, Yaroslav Kurylev, and Matti Lassas for several helpful discussions and for allowing us to use arguments from the paper \cite{DKLS}, which was in preparation simultaneously with this manuscript, involving Gaussian beam quasimodes and a reduction from the attenuated ray transform to the usual ray transform. M.S.~expresses his gratitude to the Department of Mathematics of the University of Chicago, where part of this work was carried out.

\section{Statement of results} \label{sec_results}

Our method is based on ideas developed for the anisotropic Calder\'on problem in \cite{DKSaU}, and even though much of the motivation comes from the Calder\'on problem with partial data in Euclidean domains, it is convenient to formulate our main results in the setting of manifolds. The Riemannian geometry notation used in this paper is mostly the same as in \cite{DKSaU}.

\begin{definition}
Let $(M,g)$ be a compact oriented Riemannian manifold with $C^{\infty}$ boundary, and let $n = \dim(M) \geq 3$.
\begin{enumerate}
\item[1.]
We say that $(M,g)$ is \emph{conformally transversally anisotropic} (or CTA) if 
$$
(M,g) \subset \subset (\mR \times \hat{M}_0, g), \quad g = c(e \oplus g_0)
$$
where $(\hat{M}_0,g_0)$ is some compact $(n-1)$-dimensional manifold with boundary, $e$ is the Euclidean metric on the real line, and $c$ is a smooth positive function in the cylinder $\mR \times \hat{M}_0$.

\item[2.]
We say that $(M,g)$ is \emph{admissible} if it is CTA and additionally the transversal manifold $(\hat{M}_0,g_0)$ is \emph{simple}, meaning that the boundary $\partial \hat{M}_0$ is strictly convex (the second fundamental form is positive definite) and for each $p \in \hat{M}_0$, the exponential map $\exp_p$ is a diffeomorphism from its maximal domain of definition in $T_p \hat{M}_0$ onto $\hat{M}_0$.
\end{enumerate}
\end{definition}

The uniqueness results in \cite{DKSaU} were given for admissible manifolds. In this paper we will give results both for admissible and CTA manifolds. In the main results, we will also assume that there is a compact $(n-1)$-dimensional manifold $(M_0,g_0)$ with smooth boundary such that 
\begin{equation} \label{m_mzero_mzerohat}
(M,g) \subset (\mR \times M_0, g) \subset \subset (\mR \times \hat{M}_0,g), \quad g = c(e \oplus g_0)
\end{equation}
and the following intersection is nonempty:
$$
\partial M \cap (\mR \times \partial M_0) \neq \emptyset.
$$
Under some conditions, it will be possible to ignore boundary measurements in the set $\partial M \cap (\mR \times \partial M_0)$. In the results below we will implicitly assume that the various manifolds satisfy \eqref{m_mzero_mzerohat}, and if $(M,g)$ is admissible it is also assumed that $(\hat{M}_0,g_0)$ is simple (but $(M_0,g_0)$ need not be simple, since its boundary may not be strictly convex).

Write $x = (x_1,x')$ for points in $\mR \times \hat{M}_0$ where $x_1$ is the Euclidean coordinate. The approaches of \cite{KSU}, \cite{DKSaU} are based on complex geometrical optics solutions of the form $u = e^{\tau \varphi}(m+r)$ where $\varphi$ is a special \emph{limiting Carleman weight}. We refer to \cite{DKSaU} for the definition and properties of limiting Carleman weights on manifolds. For present purposes we only mention that the functions $\varphi(x) = \pm x_1$ are natural limiting Carleman weights in the cylinder $(\mR \times \hat{M}_0, g)$.

The weight $\varphi(x) = x_1$ allows to decompose the boundary $\partial M$ as the disjoint union 
$$
\partial M = \partial M_+ \cup \partial M_- \cup \partial M_{\rm tan}
$$
where 
\begin{gather*}
\partial M_{\pm} = \{ x \in \partial M \,;\, \pm \partial_{\nu} \varphi(x) > 0 \}, \\
\partial M_{\rm tan} = \{ x \in \partial M \,;\, \partial_{\nu} \varphi(x) = 0 \}.
\end{gather*}
Here the normal derivative is understood with respect to the metric $g$. Note that $\partial_{\nu} \varphi = 0$ on $\mR \times \partial M_0$ whenever $(M_0,g_0) \subset \subset (\hat{M}_0, g_0)$. We think of $\partial M_{\rm tan}$ as being flat in one direction (the direction of the gradient of $\varphi$). For the sake of definiteness, the sets $\partial M_{\pm} = \partial M_{\pm}(\varphi)$ will refer to the weight $\varphi(x) = x_1$ in this section, but all results remain true when $\partial M_+$ and $\partial M_-$ are interchanged (this amounts to replacing the weight $x_1$ by $-x_1$).

Next we give the local results for the Calder\'on problem with partial data on manifolds. In these results we say that a unit speed geodesic $\gamma:[0,L] \to M_0$ is \emph{nontangential} if its endpoints are on $\partial M_0$, the vectors $\dot{\gamma}(0)$, $\dot{\gamma}(L)$ are nontangential, and $\gamma(t) \in M_0^{\text{int}}$ for $0 < t < L$. We also define the partial Cauchy data set as 
\begin{align*}
C_{g,q}^{\Gamma_D,\Gamma_N} = \{ (u|_{\Gamma_D}, \partial_{\nu} u|_{\Gamma_N}) &\,;\, (-\Delta_g+q) = 0 \text{ in } M, \ u \in H_{\Delta_g}(M), \\
 & \quad \supp(u|_{\partial M}) \subset \Gamma_D \}
\end{align*}
where $H_{\Delta_g}(M) = \{ u \in L^2(M) \,;\, \Delta_g u \in L^2(M) \}$ and $u|_{\partial M} \in H^{-1/2}(\partial M)$, $\partial_{\nu} u|_{\partial M} \in H^{-3/2}(\partial M)$  by the same arguments as in \cite{BU}.

To explain the results, it is convenient to think in terms of the following special case.

\begin{examplenonum}
Let $M$ be a compact manifold with boundary consisting of three parts, 
$$
M = M_{\text{left}} \cup M_{\text{mid}} \cup M_{\text{right}}
$$
where $M_{\text{mid}} = [a,b] \times M_0$ for some compact manifold $(M_0,g_0)$ with boundary, $M_{\text{left}} \subset \{ x_1 < a \} \times M_0$, and $M_{\text{right}} \subset \{ x_1 > b \} \times M_0$. We also assume that 
$$
\partial M_- = M_{\text{left}} \cap \partial M, \quad \partial M_+ = M_{\text{right}} \cap \partial M.
$$
In this case $\partial M_{\text{tan}} = [a,b] \times \partial M_0$.

The methods developed in this paper suggest that it should suffice to measure Neumann data on $\partial M_+$ for Dirichlet data supported in $\partial M_-$, with no measurements required on $\partial M_{\text{tan}}$. However, in the results below we need a part $\Gamma_a \subset \partial M_{\text{tan}}$ that is accessible to measurements, and $\Gamma_i = \partial M_{\text{tan}} \setminus \Gamma_a$ is the inaccessible part. Suppose for simplicity that 
$$
\Gamma_a = [a,b] \times E, \quad \Gamma_i = [a,b] \times (\partial M_0 \setminus E)
$$
for some nonempty open subset $E$ of $\partial M_0$.

In this setting, Theorem \ref{theorem_main_local_integrals} implies that from Neumann data measured near $\partial M_+ \cup \Gamma_a$ with Dirichlet data input near $\partial M_- \cup \Gamma_a$, one can determine certain integrals of the potential $q$ in the set 
$$
\mR \times \bigcup_{\gamma} \gamma([0,L])
$$
where the union is over all nontangential geodesics in $M_0$ with endpoints on $E$. Moreover, if the local ray transform is injective in this set in a suitable sense, then one can determine the potential in this set by Theorem \ref{theorem_main_local_result}. Theorem \ref{theorem_main_brokenrays} shows that one can go beyond this set and extract information about integrals of $q$ over all nontangential broken rays with endpoints on $E$, and Theorem \ref{theorem_main_global} gives a global uniqueness result in the case where $\partial M_{tan}$ has zero measure.
\end{examplenonum}

\begin{thm} \label{theorem_main_local_integrals}
Let $(M,g)$ be an admissible manifold as in \eqref{m_mzero_mzerohat}, and let $q_1, q_2 \in C(M)$. Let $\Gamma_i$ be a closed subset of $\partial M_{\rm tan}$, and suppose that for some nonempty open subset $E$ of $\partial M_0$ one has 
$$
\Gamma_i \subset \mR \times (\partial M_0 \smallsetminus E).
$$
Let $\Gamma_a = \partial M_{\rm tan} \smallsetminus \Gamma_i$, and assume that 
$$
C_{g,q_1}^{\Gamma_D, \Gamma_N} = C_{g,q_2}^{\Gamma_D, \Gamma_N}
$$
where $\Gamma_D$ and $\Gamma_N$ are any open sets in $\partial M$ such that  
$$
\Gamma_D \supset \partial M_- \cup \Gamma_a, \quad \Gamma_N \supset \partial M_+ \cup \Gamma_a.
$$

Given any nontangential geodesic $\gamma:[0,L] \to M_0$ with endpoints on $E$, and given any real number $\lambda$, one has 
$$
\int_0^L e^{-2\lambda t} (c(q_1-q_2))\ehat(2\lambda,\gamma(t)) \,dt = 0.
$$
Here $q_1-q_2$ is extended by zero outside $M$, and $(\,\cdot\,)\ehat$ denotes the Fourier transform in the $x_1$ variable.
\end{thm}

The previous theorem allows to conclude uniqueness of potentials in sets where the local ray transform is injective in the following sense.

\begin{definition}
Let $(M_0,g_0)$ be a compact oriented manifold with smooth boundary, and let $O$ be an open subset of $M_0$. We say that the local ray transform is injective on $O$, if any function $f \in C(M_0)$ with 
$$
\int_{\gamma} f \,dt = 0 \quad \text{for all nontangential geodesics $\gamma$ contained in $O$}
$$
must satisfy $f|_{O} = 0$.
\end{definition}

\begin{thm} \label{theorem_main_local_result}
Assume the conditions in Theorem \ref{theorem_main_local_integrals}. Then $q_1 = q_2$ in $M \cap (\mR \times O)$ for any open subset $O$ of $M_0$ such that the local ray transform is injective on $O$ and $O \cap \partial M_0 \subset E$.
\end{thm}

The local ray transform is known to be injective in the next three cases (the second case will be used in Section \ref{sec_euclidean}):
\begin{enumerate}
\item[1.]
$(M_0,g_0) = (\overline{\Omega}_0,e)$ where $\Omega_0 \subset \mR^{n-1}$ is a bounded domain with $C^{\infty}$ boundary, $e$ is the Euclidean metric, $E$ is an open subset of $\partial \Omega_0$, and $O$ is the intersection of $\overline{\Omega}_0$ with the union of all hyperplanes in $\mR^{n-1}$ that have $\partial \Omega_0 \setminus E$ on one side. The complement of this union is the intersection of half-spaces and thus convex. If the integral of $f \in C(\overline{\Omega}_0)$, extended by zero to $\mR^{n-1}$, vanishes over all line segments in $O$, then the integral over all hyperplanes that do not meet $\partial \Omega_0 \setminus E$ also vanishes, and it follows from the Helgason support theorem \cite{Helgason} that the local ray transform is injective on $O$.
\item[2.]
$(M_0,g_0) \subset \subset (\tilde{M}_0,g_0)$ are simple manifolds with real-analytic metric, and $\tilde{\mathcal{F}}$ is an open set of nontangential geodesics in $(\tilde{M}_0,g_0)$ such that any curve in $\tilde{\mathcal{F}}$ can be deformed to a point on $\partial \tilde{M}_0$ through curves in $\tilde{\mathcal{F}}$. In such a case, by a result of \cite{K} the local ray transform is injective on the set $O$ of all points in $M_0$ that lie on some geodesic in $\mathcal{\tilde{F}}$.
\item[3.]
If $\dim(M_0) \geq 3$ and if $\partial M_0$ is strictly convex at a point $p \in \partial M_0$, then $p$ has a neighborhood $O$ in $M_0$ on which the local ray transform is injective. This is a very recent result from \cite{UhlmannVasy}.
\end{enumerate}

In Theorem \ref{theorem_main_local_result}, if the nontangential geodesics with endpoints on $E$ cover a dense subset $O$ of $M_0$ and if the local ray transform is injective in $O$, we obtain a global uniqueness result stating that $q_1=q_2$ in $M$. An example of such a result under a concavity assumption is given in Section \ref{subsec_isakov_extension}. The method for proving Theorems \ref{theorem_main_local_integrals} and \ref{theorem_main_local_result} also gives a global result when the set $\partial M_{tan}$ has zero measure. Then no geometric conditions are required and one can recover the potential in all of $M$.

\begin{thm} \label{theorem_main_global}
Let $(M,g)$ be an admissible manifold and assume that $q_1, q_2 \in C(M)$. 
If $\partial M_{tan}$ has zero measure in $\partial M$, then 
$$
C_{g,q_1}^{\partial M_- , \partial M_+} = C_{g,q_2}^{\partial M_-, \partial M_+} \implies q_1 = q_2.
$$
\end{thm}

Next we wish to gather information on the potentials beyond the set that can be reached by transversal geodesics with endpoints on $E$. To do this, we will use broken geodesics in the transversal manifold that go inside $M_0$, reflect finitely many times and eventually return to $E$.

\begin{definition}
Let $(M_0,g_0)$ be a compact manifold with boundary.
\begin{itemize}
\item[(a)]
We call a continuous curve $\gamma: [a,b] \to M_0$ a \emph{broken ray} if $\gamma$ is obtained by following unit speed geodesics that are reflected according to geometrical optics (angle of incidence equals angle of reflection) whenever they hit a point of $\partial M_0$.
\item[(b)]
A broken ray $\gamma: [0,L] \to M_0$ is called \emph{nontangential} if $\dot{\gamma}(t)$ is nontangential whenever $\gamma(t) \in \partial M_0$, and additionally all points of reflection are distinct.
\end{itemize}
\end{definition}

The next theorem is a generalization of Theorem \ref{theorem_main_local_integrals} in the sense that it allows arbitrary transversal manifolds and recovers integrals over all nontangential broken rays (instead of just nontangential geodesics) with endpoints on $E$. However, it is stated with a weaker partial data condition.

\begin{thm} \label{theorem_main_brokenrays}
Let $(M,g)$ be a CTA manifold as in \eqref{m_mzero_mzerohat}, and let $q_1, q_2 \in C(M)$. Let $\Gamma_i$ be a closed subset of $\partial M_{\rm tan}$, and suppose that for some nonempty open subset $E$ of $\partial M_0$ one has 
$$
\Gamma_i \subset \mR \times (\partial M_0 \smallsetminus E).
$$
Let $\Gamma_a = \partial M_{\rm tan} \smallsetminus \Gamma_i$, and assume that 
$$
C_{g,q_1}^{\Gamma_D, \Gamma_N} = C_{g,q_2}^{\Gamma_D, \Gamma_N}
$$
where 
$$
\Gamma_D =  \Gamma_N = \Gamma
$$
for some neighborhood $\Gamma$ of the set $\overline{\partial M_+ \cup \partial M_- \cup \Gamma_a}$ in $\partial M$.

Given any nontangential broken ray $\gamma:[0,L] \to M_0$ with endpoints on $E$, and given any real number $\lambda$, one has 
$$
\int_0^L e^{-2\lambda t} (c(q_1-q_2))\ehat(2\lambda,\gamma(t)) \,dt = 0.
$$
Here $q_1-q_2$ is extended by zero outside $M$, and $(\,\cdot\,)\ehat$ denotes the Fourier transform in the $x_1$ variable.
\end{thm}

It is natural to ask whether a function in $M_0$ is determined by its integrals over broken rays with endpoints in some subset $E$ of $\partial M_0$ (that is, whether the broken ray transform is injective). Combined with Theorem \ref{theorem_main_brokenrays} and with the proof of Theorem \ref{theorem_main_local_result}, such a result would imply unique recovery of the potential in the whole manifold $M$. However, it seems that there are very few results in this direction, except for the case where $E$ is the whole boundary and the question reduces to the injectivity of the usual ray transform (see \cite{S}).

Eskin \cite{Eskin_obstacles} has proved injectivity in the case of Euclidean broken rays reflecting off several convex obstacles, with $E$ being the boundary of a smooth domain enclosing all the obstacles, if the obstacles satisfy additional restrictions (in particular the obstacles must have corner points and they cannot be smooth). Ilmavirta \cite{Ilmavirta} has recently given partial injectivity results for the broken ray transform in the Euclidean unit disc. See also \cite{Florescu}, \cite{Lozev} for related results. However, the following question seems to be open even in convex Euclidean domains except when $E=\partial M_0$.

\begin{question}
Let $(M_0,g_0)$ be a simple manifold, let $E$ be a nonempty open subset of $\partial M_0$, and assume that $f \in C(M_0)$ satisfies 
$$
\int_0^L f(\gamma(t)) \,dt = 0
$$
for all nontangential broken rays $\gamma: [0,L] \to M_0$ with endpoints on $E$. Does this imply that $f = 0$?
\end{question}

\section{The Euclidean case} \label{sec_euclidean}

In this section, we indicate some consequences of the previous results to the Calder\'on problem with partial data in Euclidean space. We assume that $\Omega \subset \mR^3$ is a bounded domain with smooth boundary equipped with the Euclidean metric $g = e$, and $q_1, q_2 \in C(\overline{\Omega})$. We also assume that 
$$
C_{q_1}^{\Gamma, \Gamma} = C_{q_2}^{\Gamma,\Gamma}
$$
where $\Gamma$ is some strict open subset of $\partial \Omega$. Write 
$$
\Gamma_i = \partial \Omega \setminus \Gamma
$$
for the inaccessible part of the boundary. The results in this section show that in cases where $\Gamma_i$ satisfies certain geometric restrictions, it is possible to conclude that 
$$
q_1 = q_2 \text{ in } \overline{\Omega} \cap (\mR \times O)
$$
where the sets $O \subset \mR^2$ will be described below.

\begin{remark}
We also obtain results for the conductivity equation by making a standard reduction to the Schr\"odinger equation. More precisely, if $\gamma_1, \gamma_2 \in C^2(\overline{\Omega})$ are positive functions such that $C_{\gamma_1}^{\Gamma, \Gamma} = C_{\gamma_2}^{\Gamma,\Gamma}$, then the corresponding DN maps satisfy 
$$
\Lambda_{\gamma_1} f|_{\Gamma} = \Lambda_{\gamma_2} f|_{\Gamma} \text{ for } f \in H^{1/2}(\partial \Omega) \text{ with } \supp(f) \subset \Gamma.
$$
Boundary determination \cite{KV}, \cite{SU_boundary} implies that 
$$
\gamma_1|_{\Gamma} = \gamma_2|_{\Gamma}, \quad \partial_{\nu} \gamma_1|_{\Gamma} = \partial_{\nu} \gamma_2|_{\Gamma}.
$$
Writing $q_j = \Delta \gamma_j^{1/2}/\gamma_j^{1/2}$, the relation 
$$
\Lambda_{q_j} f = \gamma_j^{-1/2} \Lambda_{\gamma_j}(\gamma_j^{-1/2} f) + \frac{1}{2} \gamma_j^{-1} (\partial_{\nu} \gamma_j) f\Big |_{\partial \Omega}
$$
and the above conditions imply that the DN maps $\Lambda_{q_j}$ for the Schr\"odinger equations satisfy 
$$
\Lambda_{q_1} f|_{\Gamma} = \Lambda_{q_2} f|_{\Gamma} \text{ for } f \in H^{1/2}(\partial \Omega) \text{ with } \supp(f) \subset \Gamma.
$$
Thus $C_{q_1}^{\Gamma, \Gamma} = C_{q_2}^{\Gamma,\Gamma}$, and we obtain that 
$$
q_1 = q_2 \text{ in } \overline{\Omega} \cap (\mR \times O).
$$
Write $q = q_1 = q_2$ in $\overline{\Omega} \cap (\mR \times O)$. Then $\gamma_1^{1/2}$ and $\gamma_2^{1/2}$ are both solutions of $(-\Delta+q) u = 0$ in $\overline{\Omega} \cap (\mR \times O)$ having identical Cauchy data on $\Gamma$. It follows that $\gamma_1 = \gamma_2$ in any connected component of $\overline{\Omega} \cap (\mR \times O)$ whose intersection with $\Gamma$ contains a nonempty open subset of $\partial \Omega$.
\end{remark}

In the following we will use some general facts on limiting Carleman weights from \cite{DKSaU}, where it was proved that any limiting Carleman weight in $\mR^3$ has, up to translation, rotation and scaling, one of the following six forms:
$$
x_1, \ \log \,\abs{x}, \ \text{arg}(x_1+ix_2), \ \frac{x_1}{\abs{x}^2}, \ \log \,\frac{\abs{x+e_1}^2}{\abs{x-e_1}^2}, \ \text{arg}(e^{i\theta}(x+i e_1)^2).
$$
Here $\theta \in [0,2\pi)$, and the argument function is defined by 
$$
\text{arg}(z) = 2 \arctan \frac{\im(z)}{\abs{z}+\re(z)}, \quad z \in \mC \setminus \{ t \in \mR \,;\, t \leq 0 \}.
$$
It was also proved in \cite[Section 2]{DKSaU} that if $\phi$ is a limiting Carleman weight near $(\overline{\Omega},e)$, then $\nabla_{\tilde{g}} \phi$ is a unit parallel vector field near $(\overline{\Omega}, \tilde{g})$ where 
$$
\tilde{g} = c^{-1} e, \quad c = \abs{\nabla_e \phi}_e^{-2}.
$$
Furthermore, by the proof of \cite[Lemma A.5]{DKSaU}, if $(y_1,y')$ are coordinates so that $\nabla_{\tilde{g}} \phi = \partial_{y_1}$ and if the coordinates $y'$ parametrize a $2$-dimensional manifold $S$ such that $\nabla_{\tilde{g}} \phi$ is orthogonal to $S$ with respect to the $\tilde{g}$ metric, then the metric has the form 
$$
\tilde{g}(y_1,y') = \left( \begin{array}{cc} 1 & 0 \\ 0 & \tilde{g}_0(y') \end{array} \right)
$$
where $\tilde{g}_0$ is the metric on $S$ induced by $\tilde{g}$.

\subsection{Cylindrical sets} \label{subsec_cylindrical}

This case corresponds to the limiting Carleman weight $\varphi(x) = x_1$. Suppose that $\Omega \subset \mR \times \Omega_0$, where $\Omega_0$ is a bounded domain with smooth boundary in $\mR^2$. Let $E$ be an open subset of $\partial \Omega_0$, and assume that 
$$
\Gamma_i \subset \mR \times (\partial \Omega_0 \setminus E).
$$
If $\Omega_0$ has strictly convex boundary, Theorem \ref{theorem_main_local_result} and the result of \cite{K} imply that 
$$
q_1 = q_2 \text{ in } \overline{\Omega} \cap (\mR \times O)
$$
where $O$ is the intersection of $\overline{\Omega}_0$ with the union of all lines in $\mR^2$ that have $\partial \Omega_0 \setminus E$ on one side.

The above conclusion holds true also when $\Omega_0$ does not have strictly convex boundary. To see this, let $\Omega_0 \subset \subset B \subset \subset \tilde{B}$ where $B$ and $\tilde{B}$ are balls. The extensions of the line segments in $O$ to $\tilde{B}$ form a class $\tilde{\mathcal{F}}$ such that any curve in $\tilde{\mathcal{F}}$ can be deformed to a point through curves in $\tilde{\mathcal{F}}$. It is then enough to extend $q_1-q_2$ by zero to $\mR \times \overline{B}$, and to use the proof of Theorem \ref{theorem_main_local_result} with $M_0$ replaced by $\overline{B}$ together with the result \cite{K}.

\subsection{Conical sets}

Consider the limiting Carleman weight $\phi(x) = \log \,\abs{x}$. Suppose that $\overline{\Omega} \subset \{ x_3 > 0 \}$, let $(S^2,g_0)$ be the sphere with its standard metric, let $S^2_+ = \{ \omega \in S^2 \,;\, \omega_3 > 0 \}$, and let $(M_0,g_0)$ be a compact submanifold of $(S^2_+,g_0)$ with smooth boundary. Let $E$ be an open subset of $\partial M_0$, and assume that 
$$
\Gamma_i \subset \{ r\omega \,;\, r > 0, \omega \in \partial M_0 \setminus E \}.
$$
We have $c = \abs{\nabla \phi}^{-2} = \abs{x}^2$ and $\tilde{g} = \abs{x}^{-2} e$, $\nabla_{\tilde{g}} \phi = x$. Choose coordinates so that 
$$
y_1 = \log\,\abs{x}, \quad y' = x/\abs{x}.
$$
The coordinates $y'$ parametrize the manifold $S^2$ and the metric $\tilde{g}_0$ on $S^2$ induced by $\tilde{g}$ is just the standard metric $g_0$. The discussion in the beginning of this section shows that 
$$
\tilde{g}(y_1,y') = \left( \begin{array}{cc} 1 & 0 \\ 0 & g_0(y') \end{array} \right).
$$
Now $(M_0,g_0)$ is contained in some simple submanifold $(\hat{M}_0,g_0)$ of the hemisphere $(S^2_+,g_0)$ (just remove a neighborhood of the equator). Since geodesics in $S^2_+$ are restrictions of great circles, Theorem \ref{theorem_main_local_result} and the local injectivity result \cite{K} imply as in Section \ref{subsec_cylindrical} that 
$$
q_1 = q_2 \text{ in } \overline{\Omega} \cap \{ r \omega \,;\, r > 0, \omega \in O \}
$$
where $O$ is the union of all great circle segments in $S^2_+$ such that $\partial M_0 \setminus E$ is on one side of the hyperplane containing the great circle segment.

\subsection{Surfaces of revolution}

Let $\Omega \subset \mR^3 \setminus \{ x \,;\, x_1 \leq 0 \}$, and consider the limiting Carleman weight 
$$
\phi(x) = \text{arg}(x_1+i x_2).
$$
Then $\nabla \phi = ( \frac{-x_2}{x_1^2+x_2^2}, \frac{x_1}{x_1^2+x_2^2}, 0)$ and 
$$
c = x_1^2 + x_2^2, \quad \tilde{g} = \frac{1}{x_1^2+x_2^2} e, \quad \nabla_{\tilde{g}} \phi = (-x_2,x_1,0).
$$
We make the change of coordinates valid near $\overline{\Omega}$, 
$$
y_1 = \text{arg}(x_1+i x_2), \ \ y_2 = \sqrt{x_1^2+x_2^2}, \ \ y_3 = x_3.
$$
The coordinates $y'$ parametrize the manifold $S = \{ (x_1,0,x_3) \,;\, x_1 > 0 \}$ and $\nabla_{\tilde{g}} \phi$ is orthogonal to $S$. Furthermore, we may also think of $S$ as the set $\{ (0,y_2,y_3) \,;\, y_2 > 0 \}$, and the metric on $S$ induced by $\tilde{g}$ is the hyperbolic metric $\tilde{g}_0 = \frac{1}{y_2^2} e$. The discussion in the beginning of this section shows that 
$$
\tilde{g}(y_1,y') = \left( \begin{array}{cc} 1 & 0 \\ 0 & \tilde{g}_0(y') \end{array} \right).
$$

Let $(M_0,g_0)$ be a compact submanifold of $S$ with smooth boundary, let $E$ be an open subset of $\partial M_0$. We think of $M_0$ as lying in $\{ (x_1,0,x_3) \,;\, x_1 > 0 \}$. Now, assume that 
$$
\Gamma_i \subset \{ R_{\theta}(\partial M_0 \setminus E) \,;\, \theta \in (-\pi, \pi) \}
$$
where $R_{\theta}x = (\tilde{R}_{\theta}(x_1,x_2)^t, x_3)^t$ and $\tilde{R}_{\theta}$ rotates vectors in $\mR^2$ by angle $\theta$ counterclockwise. That is, we assume that the inaccessible part $\Gamma_i$ is contained in a surface of revolution obtained by rotating the boundary curve $\partial M_0 \setminus E$.

Now, the geodesics in $S$ (and, after restriction, also in $M_0$) have either the form 
$$
(y_2(t), y_3(t)) = (R \sin t, R \cos t + \alpha)
$$
where $t \in (0,\pi)$, $R > 0$, and $\alpha \in \mR$, or the form $(y_2(t), y_3(t)) = (t,\alpha)$ where $t > 0$ and $\alpha \in \mR$ (these are not unit speed parametrizations). In the $x$ coordinates, these are either the half circles in the $\{x_2 = 0\}$ plane given by 
$$
(x_1(t), x_2(t), x_3(t)) = (R \sin t, 0, R \cos t + \alpha), \quad t \in (0,\pi)
$$
or the lines 
$$
(x_1(t), x_2(t), x_3(t)) = (t, 0, \alpha), \quad t > 0.
$$
Enclosing $M_0$ in some ball $B$ in $S$, the manifold $(\overline{B},g_0)$ is simple and Theorem \ref{theorem_main_local_result} and \cite{K} imply as in Section \ref{subsec_cylindrical} that 
$$
q_1 = q_2 \text{ in } \overline{\Omega} \cap \{ R_{\theta}(O) \,;\, \theta \in (-\pi,\pi) \}
$$
where $O$ is the union of all geodesics in $S$ that have $\partial M_0 \setminus E$ on one side.

\subsection{Other limiting Carleman weights} \label{subsec_other}

So far we have considered three of the six possible forms of limiting Carleman weights in $\mR^3$. The fourth one, $\phi(x) = \frac{x_1}{\abs{x}^2}$, is the Kelvin transform of the linear weight, and corresponds to inaccessible parts of the boundary that are Kelvin transforms of cylindrical domains. In particular, if part of the cylindrical domain is on the hyperplane $\{x_3 = 1\}$, its Kelvin transform lies on the sphere centered at $(0,0,1/2)$ with radius $1/2$, and we recover the result of Isakov \cite{I} for domains where the inaccessible part is part of a sphere. The corresponding results for the remaining two limiting Carleman weights do not seem so easy to state and we omit them.

\subsection{Extension of \cite{KSU}}

Let now $\Omega  \subset \mR^3$ be a bounded domain with smooth boundary, assume that $0$ is not in the convex hull of $\overline{\Omega}$, and let $\phi(x) = \log\,\abs{x}$. Define 
\begin{align*}
\partial \Omega_{\pm} &= \{ x \in \partial \Omega \,;\, \pm \partial_{\nu} \phi(x) > 0 \}, \\
\partial \Omega_{tan} &= \{ x \in \partial \Omega \,;\, \partial_{\nu} \phi(x) = 0 \}.
\end{align*}
It was proved in \cite{KSU} that whenever $\Gamma_D$ is a neighborhood of $\partial \Omega_- \cup \partial \Omega_{tan}$ and $\Gamma_N$ is a neighborhood of $\partial \Omega_+ \cup \partial \Omega_{tan}$, then 
$$
C_{q_1}^{\Gamma_D, \Gamma_N} = C_{q_2}^{\Gamma_D, \Gamma_N}  \implies q_1 = q_2.
$$
In particular, $\Gamma_D$ and $\Gamma_N$ always need to overlap. This result is a consequence of the reduction given above for the logarithmic weight, Theorem \ref{theorem_main_local_integrals} (the special case where $E = \partial \Omega_0$, so that $\Gamma_i = \emptyset$), and injectivity of the ray transform. If $\partial \Omega_{tan}$ has zero measure in $\partial \Omega$, then Theorem \ref{theorem_main_global} allows to improve this result: we have 
$$
C_{q_1}^{\partial \Omega_-, \partial \Omega_+} = C_{q_2}^{\partial \Omega_-, \partial \Omega_+}  \implies q_1 = q_2.
$$
In this case, the sets where Dirichlet and Neumann data are measured are disjoint, but their union covers all of $\partial \Omega$ except for a set of measure zero. The result remains true if the roles of $\partial \Omega_+$ and $\partial \Omega_-$ are changed.

\subsection{Extension of \cite{I}} \label{subsec_isakov_extension}

The results in \cite{I} stated that $C_{q_1}^{\Gamma, \Gamma} = C_{q_2}^{\Gamma,\Gamma}$ implies $q_1 = q_2$ in $\Omega$ if $\Omega \subset \{x_3 > 0\}$ and $\Gamma_i \subset \{ x_3 = 0 \}$, or if $\Omega \subset B$ for some ball $B$ and $\Gamma_i \subset \partial B$. We have already recovered these results in Sections \ref{subsec_cylindrical} and \ref{subsec_other}, since in these cases the local injectivity set $O$ is so large that the result $q_1 = q_2$ holds in all of $\Omega$. Of course, the results above also extend \cite{I} since we can conclude at least local uniqueness for potentials when the inaccessible part of the boundary satisfies a (conformal) flatness condition in only one direction, such as being part of a cylindrical set, a conical set, or a surface of revolution.

We also get global uniqueness if the local injectivity set $O$ is sufficiently large. For instance, if 
$$
\Omega \subset \mR \times \Omega_0, \quad \Gamma_i \subset \mR \times (\partial \Omega_0 \setminus E)
$$
where $\Omega_0$ is a bounded domain with smooth boundary and $E$ is a nonempty open subset of $\partial \Omega_0$, and if the lines in $\mR^2$ that have $\partial \Omega_0 \setminus E$ on one side cover a dense subset of $\Omega_0$, then we have $q_1 = q_2$ in $\Omega$. One example of this situation is if 
$$
\Omega \subset \mR \times \{ (x_2, x_3) \,;\, x_3 > \eta(x_2) \}, \quad \Gamma_i \subset \mR \times \{ (x_2, x_3) \,;\, x_3 = \eta(x_2) \}
$$
where $\eta: \mR \to \mR$ is a smooth concave function.

\section{Carleman estimate} \label{sec_carleman}

Let $(M,g)$ be a CTA manifold, so $(M,g)$ is compact with boundary and 
$$
(M,g) \subset \subset (\mR \times M_0, g), \qquad g = c(e \oplus g_0).
$$
Here $(M_0,g_0)$ is any compact $(n-1)$-dimensional manifold with boundary. We wish to prove a Carleman estimate with boundary terms for the conjugated operator $e^{\phi/h}(-\Delta_g)e^{-\phi/h}$ in $M$, where $\phi$ is the limiting Carleman weight $\phi(x) = x_1$ or $\phi(x) = -x_1$, and $h > 0$ is small. Following \cite{KSU}, it is useful to consider a slightly modified weight 
$$
\phi_{\eps} = \phi + h f_{\eps}
$$
where $f_{\eps}$ is a smooth real valued function in $M$ depending on a small parameter $\eps$, with $\eps$ independent of $h$. The convexity of $f_{\eps}$ will lead to improved lower bounds in terms of the $L^2(M)$ norms of $u$ and $h \nabla u$. On the other hand, the sign of $\partial_{\nu} \phi_{\eps}$ in the boundary term of the Carleman estimate will allow to control functions on different parts of the boundary. Of special interest is the set $\partial M_{tan}$ where $\partial_{\nu} \phi = 0$, and in this set we have 
$$
\partial_{\nu} \phi_{\eps}|_{\partial M_{tan}} = h \partial_{\nu} f_{\eps}.
$$
We would like to have $\partial_{\nu} f_{\eps} < 0$ on $\partial M_{tan}$. It is not easy to find a global convex function $f_{\eps}$ satisfying the last condition for a general set $\partial M_{tan}$. However, splitting $f_{\eps}$ to a convex part whose normal derivative vanishes on $\partial M_{tan}$ and another part which ensures the correct sign on $\partial M_{tan}$ will give the required result. We will use semiclassical conventions in the next proof, see \cite[Section 4]{DKSaU} and \cite{Z} for more details. We also write $(v,w) = (v,w)_{L^2(M)}$, $\norm{v} = \norm{v}_{L^2(M)}$, and for $\Gamma \subset \partial M$ we write $(v,w)_{\Gamma} = (v,w)_{L^2(\Gamma)}$.

\begin{prop} \label{prop_carleman_first}
Let $(M,g)$ be as above, let $\varphi(x) = \pm x_1$, and let $\kappa$ be a smooth real valued function in $M$ so that $\partial_{\nu} \kappa = -1$ on $\partial M$. Let also $q \in L^{\infty}(M)$. There are constants $\eps, C_0, h_0 > 0$ with $h_0 \leq \eps/2 \leq 1$ such that for the weight 
$$
\varphi_{\eps} = \varphi + \frac{h}{\eps} \frac{\varphi^2}{2} + h \kappa
$$
where $0 < h \leq h_0$, one has 
\begin{multline*}
\frac{h^3}{C_0} (\abs{\partial_{\nu} \phi_{\eps}} \partial_{\nu} u,\partial_{\nu} u)_{\partial M_{-}(\phi_{\eps})} + \frac{h^2}{C_0} (\norm{u}^2 + \norm{h \nabla u}^2) \\
 \leq \norm{e^{\phi/h}(-h^2 \Delta_g + h^2 q) (e^{-\phi/h} u)}^2 + h^3 ( \abs{\partial_{\nu} \phi_{\eps}} \partial_{\nu} u,\partial_{\nu} u)_{\partial M_+(\phi_{\eps})}
\end{multline*}
for any $u \in C^{\infty}(M)$ with $u|_{\partial M} = 0$.
\end{prop}
\begin{proof}
Since $\varphi(x) = \pm x_1$ is a limiting Carleman weight in a manifold strictly containing $M$, the computations in the proof of \cite[Theorem 4.1]{DKSaU} apply and we can follow that proof. First of all, note that 
$$
c^{\frac{n+2}{4}}(-\Delta_g + q)u = (-\Delta_{c^{-1} g} + q_c)(c^{\frac{n-2}{4}} u)
$$
where $q_c = cq + c^{\frac{n+2}{4}} \Delta_g(c^{-\frac{n-2}{4}})$. Thus, by replacing $q$ with another potential, we may assume that $c = 1$ so that $g = e \oplus g_0$ and $\varphi$ is a distance function in the $g$ metric, i.e.~$\abs{\nabla_g \phi}_g = 1$.

Let $P_0 = -h^2 \Delta_g$ and $P_{0,\phi_{\eps}} = e^{\phi_{\eps}/h} P_0 e^{-\phi_{\eps}/h}$. Then $P_{0,\phi_{\eps}} = A + iB$ where $A$ and $B$ are the formally self-adjoint operators 
$$
A = -h^2 \Delta_g - \abs{\nabla \phi_{\eps}}^2, \quad B = -2i \langle \nabla \phi_{\eps}, h\nabla\,\cdot\, \rangle - ih \Delta_g \phi_{\eps}.
$$
Assume $u \in C^{\infty}(M)$ and $u|_{\partial M} = 0$. We have 
\begin{align*}
\norm{P_{0,\phi_{\eps}} u}^2 &= ((A+iB)u, (A+iB)u) \\
 &= \norm{Au}^2 + \norm{Bu}^2 + i(Bu,Au) - i(Au,Bu) \\
 &= \norm{Au}^2 + \norm{Bu}^2 + (i[A,B]u,u) - i h^2 (Bu,\partial_{\nu} u)_{\partial M} \\
 &= \norm{Au}^2 + \norm{Bu}^2 + (i[A,B]u,u) - 2 h^3 ((\partial_{\nu} \phi_{\eps}) \partial_{\nu} u,\partial_{\nu} u)_{\partial M}.
\end{align*}

Define 
$$
\tilde{\phi}_{\eps}(x) = \phi + \frac{h}{\eps} \frac{\phi^2}{2}.
$$
Thus $\phi_{\eps} = \tilde{\phi}_{\eps} + h \kappa$. Let 
$$
\tilde{A} = -h^2 \Delta - \abs{\nabla \tilde{\phi}_{\eps}}^2, \quad \tilde{B} = -2i \langle \nabla \tilde{\phi}_{\eps}, h\nabla\,\cdot\, \rangle - ih \Delta \tilde{\phi}_{\eps}.
$$
Since $\Delta \phi_{\eps} = \Delta \tilde{\phi}_{\eps} + h \Delta \kappa$ and $\nabla \phi_{\eps} = \nabla \tilde{\phi}_{\eps} + h \nabla \kappa$, we have 
\begin{align*}
A = \tilde{A} + A_e, \qquad & A_e = -h^2 \abs{\nabla \kappa}^2 - 2h \langle \nabla \tilde{\phi}_{\eps}, \nabla \kappa \rangle, \\
B = \tilde{B} + B_e, \qquad & B_e = -2ih \langle \nabla \kappa, h \nabla \,\cdot\, \rangle - i h^2 \Delta \kappa.
\end{align*}
Consequently 
$$
i[A,B] = i[\tilde{A}, \tilde{B}] + i[\tilde{A},B_e] + i[A_e,\tilde{B}] + i[A_e,B_e].
$$
Recall from \cite[p.~143]{DKSaU} that 
$$
i[\tilde{A}, \tilde{B}] = \frac{4h^2}{\eps} \left( 1 + \frac{h}{\eps} \phi \right)^2 + h \tilde{B} \beta \tilde{B} + h^2 R
$$
where $\beta = (h/\eps) (1+(h/\eps) \phi)^{-2}$ and $R$ is a first order semiclassical differential operator whose coefficients are uniformly bounded with respect to $h$ and $\eps$ if we assume that $h/\eps \leq 1/2$. Consider now 
$$
i[\tilde{A},B_e] = i[-h^2 \Delta - \abs{\nabla \tilde{\phi}_{\eps}}^2, -2ih \langle \nabla \kappa, h \nabla \,\cdot\, \rangle - i h^2 \Delta \kappa].
$$
It is clear that this equals $h^2 Q$, where $Q$ is a second order semiclassical differential operator whose coefficients are uniformly bounded in $h$ and $\eps$. The terms $i[A_e,\tilde{B}]$ and $i[A_e,B_e]$ are better. We thus have 
$$
i[A,B] = \frac{4h^2}{\eps} \left( 1 + \frac{h}{\eps} \phi \right)^2 + h \tilde{B} \beta \tilde{B} + h^2 Q
$$
for some $Q$ as described above. It follows that 
$$
(i[A,B]u, u) = \frac{4h^2}{\eps} \norm{ (1 + h \phi/\eps) u}^2 + h (\tilde{B} \beta \tilde{B} u, u) + h^2(Qu,u).
$$

We will choose $h_0$ so small that $\abs{h \phi/\eps} \leq 1/2$ in $M$ for $h \leq h_0$. Since $u|_{\partial M} = 0$, integration by parts gives 
$$
\abs{ h (\tilde{B} \beta \tilde{B} u, u)} \leq C_1 \frac{h^2}{\eps} \norm{\tilde{B} u}^2
$$
Similarly 
$$
\abs{h^2(Qu,u)} \leq C_2 h^2 (\norm{u}^2 + \norm{h\nabla u}^2).
$$
Putting this information together, we get 
$$
(i[A,B]u, u) \geq \frac{h^2}{\eps} \norm{u}^2 - C_1 \frac{h^2}{\eps} \norm{\tilde{B} u}^2 - C_2 h^2 (\norm{u}^2 + \norm{h\nabla u}^2).
$$

Next we revisit the term $\norm{Au}^2$. Let $K$ be a positive constant whose value will be specified later. Since $u|_{\partial M} = 0$, integration by parts and Young's inequality give that 
\begin{align*}
h^2 \norm{h \nabla u}^2 &= h^2 (-h^2 \Delta u, u) = h^2 (Au,u) + h^2( \abs{\nabla \phi_{\eps}}^2 u, u) \\
 &\leq \frac{1}{2K} \norm{Au}^2 + \frac{K h^4}{2} \norm{u}^2 + C_3 h^2 \norm{u}^2,
\end{align*}
or 
$$
\norm{Au}^2 \geq 2K h^2 \norm{h \nabla u}^2 - K^2 h^4 \norm{u}^2 - 2 K C_3 h^2 \norm{u}^2.
$$
Also recall that $B-\tilde{B} = B_e = -2ih \langle \nabla \kappa, h \nabla \,\cdot\, \rangle - i h^2 \Delta \kappa$. Thus, 
$$
\norm{(B-\tilde{B})u}^2 \leq C_4 h^2 (\norm{u}^2 + \norm{h \nabla u}^2).
$$
Hence 
$$
\norm{\tilde{B}u}^2 \leq 2 \norm{Bu}^2 + 2 C_4 h^2 (\norm{u}^2 + \norm{h \nabla u}^2)
$$
and 
$$
\norm{Bu}^2 \geq \frac{1}{2} \norm{\tilde{B}u}^2 - C_4 h^2 (\norm{u}^2 + \norm{h \nabla u}^2).
$$
Putting our estimates together, we obtain 
\begin{align*}
 &\norm{P_{0,\phi_{\eps}} u}^2 \geq 2K h^2 \norm{h \nabla u}^2 - K^2 h^4 \norm{u}^2 - 2 K C_3 h^2 \norm{u}^2 \\
 &+ \frac{1}{2} \norm{\tilde{B}u}^2 - C_4 h^2 (\norm{u}^2 + \norm{h \nabla u}^2) + \frac{h^2}{\eps} \norm{u}^2 - C_1 \frac{h^2}{\eps} \norm{\tilde{B} u}^2 \\
 &- C_2 h^2 (\norm{u}^2 + \norm{h\nabla u}^2) - 2 h^3 ((\partial_{\nu} \phi_{\eps}) \partial_{\nu} u,\partial_{\nu} u)_{\partial M}.
\end{align*}

At this point, we choose $h_0$ so small that 
$$
C_1 \frac{h_0^2}{\eps} \leq \frac{1}{4}.
$$
We also make the choice  
$$
K = \frac{1}{\alpha \eps}
$$
where $\alpha$ is to be determined. Then for $h \leq h_0$ 
\begin{align*}
 &\norm{P_{0,\phi_{\eps}} u}^2 \geq \frac{h^2}{\eps} (\norm{u}^2  + \frac{2}{\alpha} \norm{h \nabla u}^2) - (C_2 + C_4) h^2 (\norm{u}^2 + \norm{h\nabla u}^2) \\
 &- \frac{h^2}{\eps} \frac{h^2}{\alpha^2 \eps} \norm{u}^2 - \frac{h^2}{\eps} \frac{2 C_3}{\alpha} \norm{u}^2 + \frac{1}{4} \norm{\tilde{B}u}^2 - 2 h^3 ((\partial_{\nu} \phi_{\eps}) \partial_{\nu} u,\partial_{\nu} u)_{\partial M}.
\end{align*}
Choose first $\alpha = 4 C_3$. It follows that 
\begin{align*}
 &\norm{P_{0,\phi_{\eps}} u}^2 \geq \frac{h^2}{2 \eps} \left[ 1 - 2\eps (C_2+C_4) - \frac{2h^2}{\alpha^2 \eps} \right] \norm{u}^2 \\
 &+ \frac{h^2}{\eps} \left[ \frac{2}{\alpha} - \eps (C_2 + C_4) \right] \norm{h\nabla u}^2 - 2 h^3 ((\partial_{\nu} \phi_{\eps}) \partial_{\nu} u,\partial_{\nu} u)_{\partial M}.
\end{align*}
Choose next $\eps$ so that 
$$
\eps = \min \left\{ \frac{1}{4(C_2+C_4)}, \frac{1}{\alpha(C_2+C_4)} \right\}.
$$
Finally, choose $h_0$ so that it satisfies all the restrictions made earlier ($h_0 \leq \eps/2$, $h_0 \max_{x \in M} \abs{\phi} \leq \eps/2$, and $h_0^2 \leq \eps/(4 C_1)$) and additionally 
$$
\frac{2h_0^2}{\alpha^2 \eps} \leq \frac{1}{4}.
$$
With these choices, we have 
$$
\norm{P_{0,\phi_{\eps}} u}^2 \geq \frac{h^2}{8 \eps} \norm{u}^2 + \frac{h^2}{\alpha \eps} \norm{h\nabla u}^2 - 2 h^3 ((\partial_{\nu} \phi_{\eps}) \partial_{\nu} u,\partial_{\nu} u)_{\partial M}.
$$
Adding a potential gives 
$$
\norm{P_{0,\phi_{\eps}} u}^2 \leq 2 \norm{(P_{0,\phi_{\eps}} + h^2 q) u}^2 + 2 h^4 \norm{q}_{L^{\infty}(M)}^2 \norm{u}^2.
$$
Choosing an even smaller value of  $h_0$ depending on $\norm{q}_{L^{\infty}(M)}$ if necessary, we obtain for $0 < h \leq h_0$ that 
$$
\norm{(P_{0,\phi_{\eps}} + h^2 q) u}^2 \geq \frac{h^2}{C_0} (\norm{u}^2 + \norm{h\nabla u}^2) - 2 h^3 ((\partial_{\nu} \phi_{\eps}) \partial_{\nu} u,\partial_{\nu} u)_{\partial M}.
$$
Finally, we replace $u$ by $e^{\phi^2/2\eps + \kappa} u$, where $u \in C^{\infty}(\overline{\Omega})$ and $u|_{\partial \Omega} = 0$, and use the fact that 
$$
1/C \leq e^{\phi^2/2\eps + \kappa} \leq C, \ \ \abs{\nabla(e^{\phi^2/2\eps + \kappa})} \leq C \ \text{on } M.
$$
The required estimate follows.
\end{proof}

We now pass from $\phi_{\eps}$ to $\phi$ in the boundary terms of the previous result, making use of the special properties of $\phi_{\eps}$ on $\partial M$. Note that the factor $h^4$ in the boundary term on $\{x \in \partial M \,;\, -\delta < \partial_{\nu} \phi(x) < h/3 \}$ below is weaker than the factor $h^3$ in the other boundary terms. This follows from the fact that $\partial_{\nu} \phi_{\eps} = h \partial_{\nu} \kappa = -h$ in the set where $\partial_{\nu} \phi$ vanishes, so one only has the weak lower bound.

\begin{prop} \label{prop_carleman_second}
Let $(M,g)$ be as above, let $q \in L^{\infty}(M)$, and let $\phi(x) = \pm x_1$. There exist constants $C_0, h_0 > 0$ such that whenever $0 < h \leq h_0$ and $\delta > 0$, one has 
\begin{multline*}
\frac{\delta h^3}{C_0} \norm{\partial_{\nu} u}_{L^2(\{ \partial_{\nu} \phi \leq -\delta \})}^2 + \frac{h^4}{C_0} \norm{\partial_{\nu} u}_{L^2(\{ -\delta < \partial_{\nu} \phi < h/3 \})}^2+ \frac{h^2}{C_0} (\norm{u}^2 + \norm{h \nabla u}^2) \\
 \leq \norm{e^{\phi/h}(-h^2 \Delta_g + h^2 q) (e^{-\phi/h} u)}^2 + h^3 \norm{\partial_{\nu} u}_{L^2(\{ \partial_{\nu} \phi \geq h/3 \})}^2
\end{multline*}
for any $u \in C^{\infty}(M)$ with $u|_{\partial M} = 0$.
\end{prop}
\begin{proof}
Note that 
$$
\partial_{\nu} \phi_{\eps} = \left( 1 + \frac{h}{\eps} \phi \right) \partial_{\nu} \phi + h \partial_{\nu} \kappa = \left( 1 + \frac{h}{\eps} \phi \right) \partial_{\nu} \phi - h.
$$
We choose $h_0$ so small that whenever $h \leq h_0$, one has for $x \in M$ 
$$
\frac{1}{2} \leq 1 + \frac{h}{\eps} \phi(x) \leq \frac{3}{2}.
$$
On the set where $\partial_{\nu} \phi(x) \leq -\delta$, we have  
$$
\abs{\partial_{\nu} \phi_{\eps}} \geq \delta/2.
$$
If $-\delta < \partial_{\nu} \phi < h/3$, we use the estimate 
$$
\abs{\partial_{\nu} \phi_{\eps}} \geq h/2.
$$
Moreover, $\abs{\partial_{\nu} \phi_{\eps}} \leq C_0$ on $\partial M$. Since $\{ \partial_{\nu} \phi < h/3 \} \subset \{ \partial_{\nu} \phi_{\eps} < 0 \}$ and $ \{ \partial_{\nu} \phi_{\eps} \geq 0 \} \subset \{ \partial_{\nu} \phi \geq h/3 \}$, the result follows from Proposition \ref{prop_carleman_first} after replacing $C_0$ by some larger constant.
\end{proof}

We can now obtain a solvability result from the previous Carleman estimate in a standard way by duality (see \cite{BU, KSU, NS}). There is a slight technical complication since the solution will be in $L^2$ but not in $H^1$. To remedy this we will work with the space 
$$
H_{\Delta_g}(M) = \{ u \in L^2(M) \,;\, \Delta_g u \in L^2(M) \}
$$
with norm $\norm{u}_{H_{\Delta}} = \norm{u}_{L^2} + \norm{\Delta u}_{L^2}$. As in \cite{BU}, we see that $H_{\Delta}(M)$ is a Hilbert space having $C^{\infty}(M)$ as a dense subset, and there is a well defined bounded trace operator from $H_{\Delta}(M)$ to $H^{-1/2}(\partial M)$ and a normal derivative operator from $H_{\Delta}(M)$ to $H^{-3/2}(\partial M)$. We also recall that if $u \in H_{\Delta}(M)$ and $u|_{\partial M} \in H^{3/2}(\partial M)$, then $u \in H^2(M)$.

\begin{prop} \label{prop_carleman_solvability}
Let $(M,g)$ be as above, let $q \in L^{\infty}(M)$, and let $\phi(x) = \pm x_1$. There exist constants $C_0, \tau_0 > 0$ such that when $\tau \geq \tau_0$ and $\delta > 0$, then for any $f \in L^2(M)$ and $f_- \in L^2(S_- \cup S_0)$ there exists $u \in L^2(M)$ satisfying $e^{\tau \phi} u \in H_{\Delta_g}(M)$ and $e^{\tau \phi} u|_{\partial M} \in L^2(\partial M)$ such that  
$$
e^{-\tau \phi}(-\Delta_g + q)(e^{\tau \phi} u) = f \ \ \text{in M}, \quad e^{\tau \phi} u|_{S_- \cup S_0} = e^{\tau \phi} f_-,
$$
and 
$$
\norm{u}_{L^2(M)} \leq C_0 (\tau^{-1} \norm{f}_{L^2(M)} + (\delta \tau)^{-1/2} \norm{f_-|_{S_-}}_{L^2(S_-)} + \norm{f_-|_{S_0}}_{L^2(S_0)}).
$$
Here $S_{\pm}$ and $S_0$ are the following subsets of $\partial M$:
$$
S_- = \{ \partial_{\nu} \phi \leq -\delta \}, S_0 = \{ -\delta < \partial_{\nu} \phi < 1/(3\tau) \}, S_+ = \{ \partial_{\nu} \phi \geq 1/(3\tau) \}.
$$
\end{prop}
\begin{proof}
Write $Lv = e^{\tau \phi}(-\Delta_g + \bar{q}) (e^{-\tau \phi} v)$ and $\tau = 1/h$, $\tau_0 = 1/h_0$. We rewrite the Carleman estimate of Proposition \ref{prop_carleman_second} as 
\begin{multline*}
(\delta \tau)^{1/2} \norm{\partial_{\nu} v}_{L^2(S_-)} + \norm{\partial_{\nu} v}_{L^2(S_0)} + \tau \norm{v} + \norm{\nabla v} \\
 \leq C_0 \norm{Lv} + C_0 \tau^{1/2} \norm{\partial_{\nu} v}_{L^2(S_+)}.
\end{multline*}
This is valid for any $\delta > 0$, provided that $\tau \geq \tau_0$ and $v \in C^{\infty}(M)$ with $v|_{\partial M} = 0$.

Consider the following subspace of $L^2(M) \times L^2(S_+)$:
$$
X = \{ (Lv, \partial_{\nu} v|_{S_+}) \,;\, v \in C^{\infty}(M), \ v|_{\partial M} = 0 \}.
$$
Any element of $X$ is uniquely represented as $(Lv, \partial_{\nu} v|_{S_+})$ where $v|_{\partial M} = 0$ by the Carleman estimate. Define a linear functional $l: X \to \mC$ by 
$$
l(Lv, \partial_{\nu} v|_{S_+}) = ( v, f )_{L^2(M)} - ( \partial_{\nu} v, f_- )_{L^2(S_- \cup S_0)}.
$$
By the Carleman estimate, we have 
\begin{align*}
\abs{l(Lv, \partial_{\nu} v|_{S_+})} &\leq \norm{v} \norm{f} + \norm{\partial_{\nu} v}_{L^2(S_-)} \norm{f_-}_{L^2(S_-)} \\
 &\qquad + \norm{\partial_{\nu} v}_{L^2(S_0)} \norm{f_-}_{L^2(S_0)} \\
 &\leq C_0 (\tau^{-1} \norm{f} + (\delta \tau)^{-1/2} \norm{f_-}_{L^2(S_-)} + \norm{f_-}_{L^2(S_0)}) \\
 &\qquad \times (\norm{Lv} + \tau^{1/2} \norm{\partial_{\nu} v}_{L^2(S_+)}).
\end{align*}
The Hahn-Banach theorem implies that $l$ extends to a continuous linear functional $\bar{l}: L^2(M) \times \tau^{-1/2} L^2(S_+) \to \mC$ such that 
$$
\norm{\bar{l}} \leq C_0 (\tau^{-1} \norm{f} + (\delta \tau)^{-1/2} \norm{f_-}_{L^2(S_-)} + \norm{f_-}_{L^2(S_0)}).
$$
By the Riesz representation theorem, there exist functions $u \in L^2(M)$ and $u_+ \in L^2(S_+)$ satisfying $\bar{l}(w, w_+) = (w, u)_{L^2(M)} + (w_+, u_+)_{L^2(S_+)}$. Moreover, 
\begin{multline*}
\norm{u}_{L^2(M)} + \tau^{-1/2} \norm{u_+}_{L^2(S_+)} \\
\leq C_0 (\tau^{-1} \norm{f} + (\delta \tau)^{-1/2} \norm{f_-}_{L^2(S_-)} + \norm{f_-}_{L^2(S_0)}).
\end{multline*}
If $v \in C^{\infty}(M)$ and $v|_{\partial M} = 0$, we have 
\begin{align}
(Lv, u)_{L^2(M)} + (\partial_{\nu} v, u_+)_{L^2(S_+)} &= ( v, f )_{L^2(M)} \label{lvu_solution} \\
 &\qquad - ( \partial_{\nu} v, f_- )_{L^2(S_- \cup S_0)}. \notag
\end{align}
Choosing $v$ compactly supported in $M^{\text{int}}$, it follows that $L^* u = f$, or 
$$
e^{-\tau \phi}(-\Delta_g + q) (e^{\tau \phi} u) = f \quad \text{in } M.
$$
Furthermore, $e^{\tau \phi} u \in H_{\Delta}(M)$.

If $w, v \in C^{\infty}(M)$ with $v|_{\partial M} = 0$, an integration by parts gives 
$$
(Lv, w) = -(e^{-\tau \phi} \partial_{\nu} v, e^{\tau \phi} w)_{L^2(\partial M)} + (v, L^* w).
$$
Given our solution $u$, we choose $u_j \in C^{\infty}(M)$ so that $e^{\tau \phi} u_j \to e^{\tau \phi} u$ in $H_{\Delta}(M)$. Applying the above formula with $w = u_j$ and taking the limit, we see that 
$$
(Lv, u) = -(e^{-\tau \phi} \partial_{\nu} v, e^{\tau \phi} u)_{L^2(\partial M)} + (v, L^* u)
$$
for $v \in C^{\infty}(M)$ with $v|_{\partial M} = 0$. Combining this with \eqref{lvu_solution}, using that $L^* u = f$, gives 
$$
( \partial_{\nu} v, f_- )_{L^2(S_- \cup S_0)} + (\partial_{\nu} v, u_+)_{L^2(S_+)} = (e^{-\tau \phi} \partial_{\nu} v, e^{\tau \phi} u)_{L^2(\partial M)}.
$$Since $\partial_{\nu} v$ can be chosen arbitrarily, it follows that 
$$
e^{\tau \phi} u|_{S_- \cup S_0} = e^{\tau \phi} f_-, \quad e^{\tau \phi} u|_{S_+} = e^{\tau \phi} u_+.
$$
We also see that $e^{\tau \phi} u|_{\partial M} \in L^2(\partial M)$.
\end{proof}

\section{Reflection approach} \label{sec_reflection}

In the previous section, we employed Carleman estimates and duality to obtain a solvability result (Proposition \ref{prop_carleman_solvability}) that will be used to produce correction terms in complex geometrical optics solutions with prescribed behavior on parts of the boundary. In this section we give an alternative approach to the construction of correction terms vanishing on parts of the boundary. The method is based on a reflection argument. We extend the method of \cite{I}, which dealt with inaccessible parts that are part of a hyperplane, to the case of inaccessible parts that are part of the graph of a function independent of one of the variables. The results are less general than the ones in Section \ref{sec_carleman}, and for simplicity will only be stated for domains in $\mR^3$ with Euclidean metric, but on the other hand the method is constructive and is based on direct Fourier arguments in the spirit of \cite{KSaU}, \cite{KSaU_reconstruction}.

Let $\Omega \subset \mR^3$ be a bounded domain with smooth boundary, and assume that 
$$
\Omega \subset \mR \times \{ (x_2,x_3) \,;\, x_3 > \eta(x_2) \}
$$
where $\eta: \mR \to \mR$ is a smooth function. Also assume that $\Gamma_0$ is a closed subset of $\partial \Omega$ such that 
$$
\Gamma_0 \subset \mR \times \{ (x_2,x_3) \,;\, x_3 = \eta(x_2) \}.
$$
We will show that if one has access to suitable amplitudes of complex geometrical optics solutions that vanish on $\Gamma_0$, then it is possible to produce correction terms that also vanish on $\Gamma_0$.

\begin{prop} \label{prop_cgo_reflection}
Let $\Omega$ and $\Gamma_0$ be as above, and let $q \in L^{\infty}(\Omega)$. There are $C_0$, $\tau_0 > 0$ such that for any $\tau$ with $\abs{\tau} \geq \tau_0$ and for any $m \in H^2(\Omega)$ with $m|_{\Gamma_0} = 0$, the equation $(-\Delta+q)u = 0$ in $\Omega$ has a solution $u \in H^2(\Omega)$ of the form 
$$
u = e^{-\tau x_1}(m + r)
$$
such that $r|_{\Gamma_0} = 0$ and 
$$
\norm{r}_{L^2(\Omega)} \leq \frac{C_0}{\abs{\tau}} \norm{e^{\tau x_1}(-\Delta+q)(e^{-\tau x_1} m)}_{L^2(\Omega)}.
$$
\end{prop}

The proof involves a reflection argument that reduces the construction of the correction term to the problem of solving a conjugated equation with anisotropic metric, 
$$
e^{\tau x_1}(-\Delta_{\hat{g}} + \hat{q})(e^{-\tau x_1} \hat{r}) = \hat{f} \quad \text{in } \mR \times \hat{\Omega}_0
$$
where $\hat{\Omega}_0 \subset \mR^2$ is a bounded open set and $\hat{g}$ is a metric of the form 
$$
\hat{g}(y_1,y') = \left( \begin{array}{cc} 1 & 0 \\ 0 & \hat{g}_0(y') \end{array} \right),
$$
and where $g_0$ is smooth for $y_3 \neq 0$ but only Lipschitz continuous across $\{ y_3 = 0 \}$. In three and higher dimensions, it is not known how to handle equations of this type with general Lipschitz coefficients in the second order part (the case of $C^1$ coefficients, and also Lipschitz coefficients with a smallness condition, is considered in \cite{HT}). However, in our case the singularity of $\hat{g}$ only appears in the lower right block $\hat{g}_0$, and this turns out not to be a problem.

The following is an analogue of \cite[Proposition 4.1]{KSaU}, the main difference being that the transversal metric is only Lipschitz. (With correct definitions, one could easily deal with $L^{\infty}$ transversal metrics as well, but then the solution would only be in $H^1_{-\delta}(T)$.) Here we write $(x_1,x')$ for coordinates in $T = \mR \times M_0$, and for $\delta \in \mR$ we consider the spaces 
$$
\norm{f}_{L^2_{\delta}(T)} = \norm{\br{x_1}^{\delta} f}_{L^2(T)}, \quad \norm{f}_{H^1_{\delta}(T)} = \norm{f}_{L^2_{\delta}(T)} + \norm{df}_{L^2_{\delta}(T)}
$$
with $\br{t} = (1+t^2)^{1/2}$, and similarly for $H^2_{\delta}(T)$. We also write $\text{Spec}(-\Delta_{g_0})$ for the set of Dirichlet eigenvalues of the Laplace-Beltrami operator $-\Delta_{g_0}$ in $(M_0,g_0)$.

\begin{prop} \label{prop_correction_fourier_nonsmooth}
Let $T = \mR \times M_0$ with metric $g = e \oplus g_0$, where $(M_0,g_0)$ is a compact oriented manifold with smooth boundary and $g_0$ is a Lipschitz continuous Riemannian metric on $M_0$. Given any $q \in L^{\infty}_{comp}(T)$ and any $\delta > 1/2$, there are constants $C_0, \tau_0 > 0$ such that whenever 
$$
\abs{\tau} \geq \tau_0 \text{ and } \tau^2 \notin \mathrm{Spec}(-\Delta_{g_0}),
$$
the equation 
$$
e^{\tau x_1}(-\Delta_g + q)(e^{-\tau x_1} r) = f \quad \text{in } T
$$
has a unique solution $r \in H^1_{-\delta}(T)$ with $r|_{\partial T} = 0$ for any $f \in L^2_{\delta}(T)$. Moreover, $r \in H^2_{-\delta}(T)$, and one has the bounds 
$$
\norm{r}_{L^2_{-\delta}(T)} \leq \frac{C_0}{\abs{\tau}} \norm{f}_{L^2_{\delta}(T)}, \quad \norm{r}_{H^1_{-\delta}(T)} \leq C_0 \norm{f}_{L^2_{\delta}(T)}.
$$
\end{prop}
\begin{proof}
The proof is almost exactly the same as the proof of \cite[Proposition 4.1]{KSaU}, and we only give the main idea. Since $\Delta_g = \partial_{x_1}^2 + \Delta_{g_0}$, the equation that we need to solve is 
$$
(-\partial_{x_1}^2 + 2 \tau \partial_{x_1} - \Delta_{g_0} - \tau^2 + q) r = f \quad \text{in } T.
$$
It is enough to consider $q = 0$. The standard argument based on weak solutions shows that even when $g_0$ has very little regularity, there is an orthonormal basis of $L^2(M_0)$ consisting of Dirichlet eigenfunctions of $-\Delta_{g_0}$, 
$$
-\Delta_{g_0} \phi_l = \lambda_l \phi_l \ \ \text{in } M_0, \quad \phi_l \in H^1_0(M_0),
$$
where $0 < \lambda_1 \leq \lambda_2 \leq \lambda_3 \leq \ldots \to \infty$ are the Dirichlet eigenvalues of $-\Delta_{g_0}$ in $M_0$.

Considering the partial Fourier expansions 
$$
r(x_1,x') = \sum_{l=1}^{\infty} \tilde{r}(x_1,l) \phi_l(x'), \quad f(x_1,x') = \sum_{l=1}^{\infty} \tilde{f}(x_1,l) \phi_l(x'),
$$
it is enough to solve 
$$
(-\partial_{x_1}^2 + 2 \tau \partial_{x_1} + \lambda_l - \tau^2) \tilde{r}(\,\cdot\,,l) = \tilde{f}(\,\cdot\,,l) \quad \text{in } \mR \text{ for all $l$.}
$$
The condition $\tau^2 \notin \mathrm{Spec}(-\Delta_{g_0})$ allows to solve these ordinary differential equations by the Fourier transform as in \cite[Section 4]{KSaU}, and the estimates given there imply that one obtains a unique solution $r \in H^1_{-\delta}(T)$ with $r|_{\partial T} = 0$ satisfying the required bounds. Elliptic $H^2$ regularity also works with Lipschitz $g_0$, and the argument in \cite[Section 4]{KSaU} gives that $r \in H^2_{-\delta}(T)$.
\end{proof}

\begin{proof}[Proof of Proposition \ref{prop_cgo_reflection}]
We begin by flattening $\Gamma_0$ via the map 
$$
\Phi: \mR^3 \to \mR^3, \ \ (x_1,x_2,x_3) \mapsto (x_1, x_2, x_3 - \eta(x_2)).
$$
Let $\tilde{\Omega} = \Phi(\Omega)$, write $y$ for coordinates in $\tilde{\Omega}$, and let $R$ be the reflection 
$$
R(y_1,y_2,y_3) = (y_1,y_2,-y_3).
$$
Note that $\tilde{\Omega} \subset \{ y_3 > 0 \}$. Consider the reflected domain $\tilde{\Omega}^* = R(\tilde{\Omega})$, so $\tilde{\Omega}^* \subset \{ y_3 < 0 \}$, and let $U$ the double domain $\tilde{\Omega} \cup \Phi(\Gamma_0)^{\text{int}} \cup \tilde{\Omega}^*$.

Let $\Psi = \Phi^{-1}$, let $\tilde{g} = \Psi^* e$ be the metric in $\tilde{\Omega}$ that is the pullback of the Euclidean metric in $\Omega$, let $\tilde{q} = \Psi^* q$, and let $\tilde{m} = \Psi^* m$. In the double domain $U$, we use even reflection to define the quantities 
$$
\hat{g} = \left\{ \begin{array}{cl} \tilde{g}, & y_3 > 0, \\ R^* \tilde{g}, & y_3 < 0, \end{array} \right. \quad \hat{q} = \left\{ \begin{array}{cl} \tilde{q}, & y_3 > 0, \\ R^* \tilde{q}, & y_3 < 0, \end{array} \right.
$$
and odd reflection to define the amplitude 
$$
\hat{m} = \left\{ \begin{array}{cl} \frac{1}{2} \tilde{m}, & y_3 > 0, \\ -\frac{1}{2} R^* \tilde{m}, & y_3 < 0. \end{array} \right.
$$
Since the flattening map $\Phi$ leaves $x_1$ intact, we have 
$$
\hat{g}(y_1,y') = \left( \begin{array}{cc} 1 & 0 \\ 0 & \hat{g}_0(y') \end{array} \right)
$$
where $\hat{g}_0$ is a Lipschitz continuous metric only depending on $y_2$ and $y_3$. (In fact, $\tilde{g}$ and $\tilde{g}_0$ are well defined in $\{ y_3 > 0 \}$ by the flattening map $\Phi$ and the Euclidean metric in $\{ x_3 > \eta(x_2) \}$.) Also, $\hat{q} \in L^{\infty}(U)$, and $\hat{m} \in H^2(U)$ by the boundary condition $m|_{\Gamma_0} = 0$ and by the properties of odd reflection.

We wish to find $\hat{r} \in H^1(U)$ satisfying 
$$
e^{\tau x_1}(-\Delta_{\hat{g}} + \hat{q})(e^{-\tau x_1} \hat{r}) = \hat{f}
$$
where $\hat{f} = -e^{\tau x_1}(-\Delta_{\hat{g}} + \hat{q})(e^{-\tau x_1} \hat{m})$. Now 
\begin{align*}
\norm{\hat{f}}_{L^2(U)} &= \norm{\hat{f}}_{L^2(\tilde{\Omega})} + \norm{\hat{f}}_{L^2(\tilde{\Omega}^*)} \\
 &= \norm{\Psi^*(e^{\tau x_1}(-\Delta + q)(e^{-\tau x_1} m)}_{L^2(\tilde{\Omega})} \\
 &\qquad + \norm{R^* \Psi^*(e^{\tau x_1}(-\Delta + q)(e^{-\tau x_1} m)}_{L^2(\tilde{\Omega}^*)} \\
 &\leq C \norm{e^{\tau x_1}(-\Delta + q)(e^{-\tau x_1} m)}_{L^2(\Omega)}.
\end{align*}
Choose a bounded open set $\hat{\Omega}_0 \subset \mR^2$ such that 
$$
U \subset \subset \mR \times \hat{\Omega}_0,
$$
and let $\hat{g}_0$ be the metric in $\hat{\Omega}_0$ that is the even extension of $\tilde{g}_0$ from $\{ y_3 > 0 \}$ to $\hat{\Omega}_0$. Then $\hat{g}_0$ is smooth for $y_3 \neq 0$ and Lipschitz continuous across $\{ y_3 = 0 \}$. Extending $\hat{g}$ to $\mR \times \hat{\Omega}_0$ using the block structure and extending $\hat{q}$ and $\hat{f}$ by zero to $\mR \times \hat{\Omega}_0$, it is enough to find a solution $\hat{r} \in H^2_{loc}(\mR \times \hat{\Omega}_0)$ of the equation 
\begin{equation}
e^{\tau x_1}(-\Delta_{\hat{g}} + \hat{q})(e^{-\tau x_1} \hat{r}) = \hat{f} \quad \text{in } \mR \times \hat{\Omega}_0.
\end{equation}
Such a solution may be found by Proposition \ref{prop_correction_fourier_nonsmooth}, and denoting by $\hat{r}$ its restriction to $U$ we have 
$$
\norm{\hat{r}}_{L^2(U)} \leq \frac{C}{\abs{\tau}} \norm{\hat{f}}_{L^2(U)}.
$$

Define now 
$$
\hat{u} = e^{-\tau x_1}(\hat{m} + \hat{r}) \quad \text{in } U
$$
and 
$$
\tilde{u} = \hat{u} - R^* \hat{u} \quad \text{in } \tilde{\Omega}.
$$
Then $(-\Delta_{\hat{g}} + \hat{q}) \hat{u} = 0$ in $U$, and $(-\Delta_{\tilde{g}} + \tilde{q}) \tilde{u} = 0$ in $\tilde{\Omega}$ by the definition of $\hat{g}$ and $\hat{q}$ and using that $\hat{u} \in H^2(U)$. We also have 
$$
\tilde{u} = e^{-\tau x_1}(\hat{m} - R^* \hat{m} + \hat{r} - R^* \hat{r}) \quad \text{in } \tilde{\Omega}.
$$
But here $\hat{m} - R^* \hat{m}|_{\tilde{\Omega}} = \tilde{m}$ by the definition of $\hat{m}$. Consequently, if we define $u = \Phi^* \tilde{u}$, then $(-\Delta+q)u = 0$ in $\Omega$ and 
$$
u = e^{-\tau x_1}(m + r) \quad \text{in } \Omega,
$$
where $r = \Phi^*(\hat{r} - R^* \hat{r})$ satisfies 
\begin{align*}
\norm{r}_{L^2(\Omega)} &\leq C \norm{\hat{r}}_{L^2(U)} \leq \frac{C}{\abs{\tau}} \norm{\hat{f}}_{L^2(U)} \\
 &\leq \frac{C}{\abs{\tau}} \norm{e^{\tau x_1}(-\Delta + q)(e^{-\tau x_1} m)}_{L^2(\Omega)}.
\end{align*}
This proves the result.
\end{proof}

Note how the odd reflection of the amplitude $m$ in the proof ensured that the solution obtained by reflection is not the zero solution. We also remark that under certain conditions, the arguments in Sections \ref{sec_simple} and \ref{sec_gaussian} allow to construct amplitudes $m$ vanishing on a part $\Gamma_0$ as above.

\section{Local uniqueness on simple manifolds} \label{sec_simple}

In this section we prove Theorems \ref{theorem_main_local_integrals}--\ref{theorem_main_global}. In these results the transversal manifold is assumed to be simple and we only use non-reflected geodesics. This case already illustrates the main features of the approach, and we can use a quasimode construction that is much easier than the Gaussian beam one used for non-simple transversal manifolds and reflected geodesics.

The first observation is the usual integral identity.

\begin{prop} \label{prop_integralidentity}
If $\Gamma_D, \Gamma_N \subset \partial M$ are open and if $C_{g,q_1}^{\Gamma_D, \Gamma_N} = C_{g,q_2}^{\Gamma_D, \Gamma_N}$, then 
$$
\int_M (q_1-q_2) u_1 u_2 \,dV_g = 0
$$
for any $u_j \in H_{\Delta_g}(M)$ satisfying $(-\Delta_g + q_j) u_j = 0$ in $M$ and 
$$
\supp(u_1|_{\partial M}) \subset \Gamma_D, \quad \supp(u_2|_{\partial M}) \subset \Gamma_N.
$$
\end{prop}
\begin{proof}
Let $u_j$ be as stated. Since $C_{g,q_1}^{\Gamma_D, \Gamma_N} = C_{g,q_2}^{\Gamma_D, \Gamma_N}$, there is a function $\tilde{u}_2 \in H_{\Delta}(M)$ with $(-\Delta + q_2) \tilde{u}_2 = 0$ in $M$, $\supp(\tilde{u}_2|_{\partial M}) \subset \Gamma_D$, and 
$$
(u_1|_{\Gamma_D}, \partial_{\nu} u_1|_{\Gamma_N}) = (\tilde{u}_2|_{\Gamma_D}, \partial_{\nu} \tilde{u}_2|_{\Gamma_N}).
$$
Using that $u_1$, $u_2$ and $\tilde{u}_2$ are solutions, we have 
\begin{align*}
\int_M (q_1-q_2) u_1 u_2 \,dV &= \int_M \left[ (\Delta u_1) u_2 - u_1 (\Delta u_2) \right] \,dV \\
 &= \int_M \left[ (\Delta (u_1-\tilde{u}_2)) u_2 - (u_1-\tilde{u}_2) (\Delta u_2) \right] \,dV.
\end{align*}
Now $u_1-\tilde{u}_2|_{\partial M} = 0$, so in fact $u_1-\tilde{u}_2 \in H^2(M)$ by the properties of the space $H_{\Delta}(M)$. Recall also that $C^{\infty}(M)$ is dense in $H_{\Delta}(M)$ and that $u_2|_{\partial M} \in H^{-1/2}(\partial M)$ and $\partial_{\nu} u_2|_{\partial M} \in H^{-3/2}(\partial M)$. These facts make it possible to integrate by parts, and we obtain that 
$$
\int_M (q_1-q_2) u_1 u_2 \,dV = \int_{\partial M} \left[ (\partial_{\nu} (u_1-\tilde{u}_2)) u_2 - (u_1-\tilde{u}_2) (\partial_{\nu} u_2) \right] \,dS
$$
in the weak sense. The last expression vanishes since $\partial_{\nu} (u_1-\tilde{u}_2)|_{\Gamma_N} = 0$ and $\supp(u_2|_{\partial M}) \subset \Gamma_N$.
\end{proof}

The next result will be used to pass from the metric $g = c(e \oplus g_0)$ to the slightly simpler metric $\tilde{g} = e \oplus g_0$.

\begin{lemma} \label{lemma_conformal_scaling_solutions}
Let $c$ be a smooth positive function in $M$. Then $u \in H_{\Delta_g}(M)$ satisfies $(-\Delta_g+q)u = 0$ in $M$ if and only if $\tilde{u} \in H_{\Delta_{\tilde{g}}}(M)$ satisfies $(-\Delta_{\tilde{g}}+\tilde{q}) \tilde{u} = 0$ in $M$, where 
$$
\tilde{g} = c^{-1} g, \quad \tilde{u} = c^{\frac{n-2}{4}} u, \quad \tilde{q} = c(q - c^{\frac{n-2}{4}} \Delta_g(c^{-\frac{n-2}{4}})).
$$
\end{lemma}
\begin{proof}
This follows from the identity for $v \in C^{\infty}(M)$, 
$$
c^{\frac{n+2}{4}} (-\Delta_g + q) (c^{-\frac{n-2}{4}} v) = (-\Delta_{c^{-1} g} + c(q-c^{\frac{n-2}{4}} \Delta_g(c^{-\frac{n-2}{4}})))v,
$$
upon approximating $u$ or $\tilde{u}$ by smooth functions.
\end{proof}

\begin{proof}[Proof of Theorem \ref{theorem_main_local_integrals}]
Let $\tilde{g} = e \oplus g_0$ and $\tilde{q}_j = c(q_j - c^{\frac{n-2}{4}} \Delta_g(c^{-\frac{n-2}{4}}))$. Let $\lambda$ be a fixed real number, and consider the complex frequency 
$$
s = \tau + i\lambda
$$
where $\tau > 0$ will be large. We look for solutions 
\begin{align*}
\tilde{u}_1 &= e^{-s x_1}(v_s(x') + r_1), \\
\tilde{u}_2 &= e^{s x_1}(v_s(x') + r_2),
\end{align*}
of the equations $(-\Delta_{\tilde{g}} + \tilde{q}_1) \tilde{u}_1 = 0$, $(-\Delta_{\tilde{g}} + \overline{\tilde{q}_2}) \tilde{u}_2 = 0$ in $M$. Here $v_s$ will be a quasimode for the Laplacian in $(M_0,g_0)$ that concentrates near the given geodesic $\gamma$. Next we will construct a suitable solution $\tilde{u}_1$, and the case of $\tilde{u}_2$ will be analogous.

Since $\Delta_{\tilde{g}} = \partial_1^2 + \Delta_{g_0}$, the function $\tilde{u}_1$ is a solution if and only if 
\begin{equation} \label{esxone_equation_easy}
e^{s x_1} (-\Delta_{\tilde{g}} + \tilde{q}_1) (e^{-s x_1} r_1) = -(-\Delta_{g_0} + \tilde{q}_1 - s^2) v_s(x') \quad \text{in } M.
\end{equation}
We want to choose $v_s \in C^{\infty}(M_0)$ to satisfy 
\begin{equation} \label{vs_easy_quasimode_bounds}
\norm{v_s}_{L^2(M_0)} = O(1), \quad \norm{(-\Delta_{g_0} - s^2) v_s}_{L^2(M_0)} = O(1)
\end{equation}
as $\tau \to \infty$. Looking for $v_s$ in the form 
$$
v_s = e^{is\psi} a
$$
where $\psi, a \in C^{\infty}(M_0)$, a direct computation shows that 
\begin{multline*}
(-\Delta_{g_0} - s^2) v_s \\
 = e^{is\psi} \left( s^2 \left[ \abs{d\psi}_{g_0}^2 - 1 \right] a - is \left[ 2 \langle d\psi, d\,\cdot\, \rangle_{g_0} + \Delta_{g_0} \psi \right] a - \Delta_{g_0} a \right).
\end{multline*}

Since $(M_0,g_0)$ is simple, it is easy to find $\psi$ and $a$ so that the expressions in brackets will vanish and that the resulting quasimode $v_s$ will concentrate near the geodesic $\gamma$. To do this, let $(\hat{M}_0,g_0)$ be a simple manifold that is slightly larger than $(M_0,g_0)$, extend $\gamma$ as a geodesic in $\hat{M}_0$, and choose $\eps > 0$ such that $\gamma|_{(-2\eps,0) \cup (L,L+2\eps)}$ stays in $\hat{M}_0 \setminus M_0$ (this is possible since $\gamma$ is nontangential). Let $\omega = \gamma(-\eps) \in \hat{M}_0 \setminus M_0$, and let $(r,\theta)$ be polar normal coordinates in $(M_0,g_0)$ with center $\omega$. Then $\gamma$ corresponds to the curve $r \mapsto (r,\theta_0)$ for some fixed $\theta_0 \in S^{n-2}$. We will choose 
\begin{gather*}
\psi(r,\theta) = r, \\
a(r,\theta) = \abs{g_0(r,\theta)}^{-1/4} b(\theta)
\end{gather*}
where $\abs{g_0}$ is the determinant of $g_0$, and $b$ is a fixed function in $C^{\infty}(S^{n-2})$ that is supported so close to $\theta_0$ such that $v_s|_{\partial M_0 \setminus E} = 0$. With these choices, we have as in \cite{DKSaU} 
$$
(-\Delta_{g_0} - s^2) v_s = - e^{is\psi} \Delta_{g_0} a.
$$
Thus $v_s$ satisfies the estimates \eqref{vs_easy_quasimode_bounds}, and also the estimate 
$$
\norm{v_s}_{L^{\infty}(M_0)} = O(1).
$$

We now go back to \eqref{esxone_equation_easy}, and look for a solution in the form $r_1 = e^{i\lambda x_1} r_1'$ where $r_1'$ satisfies 
\begin{equation} \label{roneprime_equation}
e^{\tau x_1} (-\Delta_{\tilde{g}} + \tilde{q}_1) (e^{-\tau x_1} r_1') = f \quad \text{in } M
\end{equation}
with 
$$
f = -e^{-i\lambda x_1}(-\Delta_{g_0} + \tilde{q}_1 - s^2) v_s(x').
$$
We also want to arrange that $\supp(\tilde{u}_1|_{\partial M}) \subset \Gamma_D$ where $\Gamma_D \supset \partial M_- \cup \Gamma_a$. For this purpose, let $\delta > 0$ be a small number to be fixed later, let $S_{\pm}$ and $S_0$ be the sets in Proposition \ref{prop_carleman_solvability} with Carleman weight $\varphi(x) = -x_1$, define 
\begin{gather*}
V^{\delta} = \{ x \in S_- \cup S_0 \,;\, \text{dist}_{\partial M}(x, \Gamma_i) < \delta \text{ or } x \in \partial M_+ \}, \\
\Gamma_a^{\delta} = (S_- \cup S_0) \setminus V_{\delta},
\end{gather*}
and impose the boundary condition 
\begin{equation} \label{roneprime_boundarycondition}
e^{\tau \phi} r_1'|_{S_- \cup S_0} = e^{\tau \phi} f_-
\end{equation}
where 
$$
f_- = \left\{ \begin{array}{cl} -e^{-i\lambda x_1} v_s(x'), & \quad \text{on } V^{\delta}, \\ 0, & \quad \text{on } \Gamma_a^{\delta}. \end{array} \right.
$$
Note that $\partial M_+ \cup \partial M_{tan}$ (these sets refer to the weight $x_1$) is in the interior of $S_- \cup S_0$ in $\partial M$.

We have seen that $\norm{f}_{L^2(M)} = O(1)$ as $\tau \to \infty$. We also have 
$$
f_-|_{\partial M_{tan}} = 0,
$$
since $f_-|_{\Gamma_a^{\delta} \cap \partial M_{tan}} = 0$ by definition and $f_-|_{\partial M_{tan} \cap V_{\delta}} = 0$ for sufficiently small $\delta > 0$ by the construction of $v_s$ and using that $\Gamma_i \subset \mR \times (\partial M_0 \setminus E)$. Since $\norm{f_-}_{L^{\infty}(S_{-} \cup S_0)} \lesssim 1$, we have 
$$
\norm{f_-}_{L^2(S_-)} \lesssim \sigma( \{ \partial_{\nu} x_1 \geq \delta \})
$$
and 
$$
\norm{f_-}_{L^2(S_0)} \lesssim \sigma( \{ -1/(3\tau) < \partial_{\nu} x_1 < 0 \} \cup \{ 0 < \partial_{\nu} x_1 < \delta \} ),
$$
where $\sigma$ is the surface measure on $\partial M$. It follows from Proposition \ref{prop_carleman_solvability} that the equation \eqref{roneprime_equation} has a solution $r_1'$ satisfying the boundary condition \eqref{roneprime_boundarycondition}, and having the estimate 
\begin{multline*}
\norm{r_1'}_{L^2(M)} \lesssim \tau^{-1} + (\delta \tau)^{-1/2}\sigma( \{ \partial_{\nu} x_1 \geq \delta \}) \\
 + \sigma( \{ -1/(3\tau) < \partial_{\nu} x_1 < 0 \}) + \sigma(\{ 0 < \partial_{\nu} x_1 < \delta \} ).
\end{multline*}
The implied constants in the previous inequality are independent of $\tau$ and $\delta$. By the basic properties of measures, for some constant $C_0 > 0$ we have 
$$
\norm{r_1'}_{L^2(M)} \leq C_0 \left[ \tau^{-1} + (\delta \tau)^{-1/2} + o_{\tau \to \infty}(1) + o_{\delta \to 0}(1) \right].
$$
Given $\eps > 0$, we first choose $\delta$ so that $C_0 o_{\delta \to 0}(1) \leq \eps/2$. After this, we choose $\tau > 0$ so large that $C_0 (\tau^{-1} + (\delta \tau)^{-1/2} + o_{\tau \to \infty}(1)) \leq \eps/2$. This shows that 
$$
\lim_{\tau \to \infty} \norm{r_1'(\,\cdot\, ; \tau)}_{L^2(M)} = 0.
$$

Choosing $r_1'$ as described above and choosing $r_1 = e^{i\lambda x_1} r_1'$, we have produced a solution $\tilde{u}_1 \in H_{\Delta_{\tilde{g}}}(M)$ of the equation $(-\Delta_{\tilde{g}} + \tilde{q}_1) \tilde{u}_1 = 0$ in $M$, having the form 
$$
\tilde{u}_1 = e^{-s x_1}(v_s(x') + r_1)
$$
and satisfying 
$$
\supp(\tilde{u}_1|_{\partial M}) \subset \Gamma_D
$$
and $\norm{r_1}_{L^2(M)} = o(1)$ as $\tau \to \infty$. Repeating this construction for the Carleman weight $\phi(x) = x_1$, we obtain a solution $\tilde{u}_2 \in H_{\Delta_{\tilde{g}}}(M)$ of the equation $(-\Delta_{\tilde{g}} + \overline{\tilde{q}_2}) \tilde{u}_2 = 0$ in $M$, having the form 
$$
\tilde{u}_2 = e^{s x_1}(v_s(x') + r_2)
$$
and satisfying 
$$
\supp(\tilde{u}_2|_{\partial M}) \subset \Gamma_N
$$
and $\norm{r_2}_{L^2(M)} = o(1)$ as $\tau \to \infty$.

Writing $u_j = c^{-\frac{n-2}{4}} \tilde{u}_j$, Lemma \ref{lemma_conformal_scaling_solutions} shows that $u_j \in H_{\Delta_g}(M)$ are solutions of $(-\Delta_g + q_1) u_1 = 0$ and $(-\Delta_g + \overline{q_2}) u_2 = 0$ in $M$. Then Proposition \ref{prop_integralidentity} implies that 
$$
\int_M (q_1-q_2) u_1 \overline{u_2} \,dV_g = 0.
$$
We extend $q_1-q_2$ by zero to $\mR \times M_0$. Inserting the expressions for $u_j$, and using that $dV_g = c^{n/2} \,dx_1 \,dV_{g_0}(x')$, we obtain 
$$
\int_{M_0} \int_{-\infty}^{\infty} (q_1-q_2) c e^{-2i\lambda x_1}(\abs{v_s(x')}^2 + v_s \overline{r_2} + \overline{v_s} r_1 + r_1 \overline{r_2}) \,dx_1 \,dV_{g_0}(x') = 0.
$$
Since $\norm{r_j}_{L^2(M)} = o(1)$ as $\tau \to \infty$ and since $dV_{g_0} = \abs{g_0}^{1/2} \,dr \,d\theta$ in the $(r,\theta)$ coordinates, it follows that 
$$
\int_{S^{n-2}} \int_0^{\infty} e^{-2\lambda r}(c(q_1-q_2))\ehat(2\lambda,r,\theta) \abs{b(\theta)}^2 \,dr \,d\theta = 0.
$$
Varying $b$ in $C^{\infty}(S^{n-2})$ so that the support of $b$ is very close to $\theta_0$, this implies that 
$$
\int_0^{\infty} e^{-2\lambda r} (c(q_1-q_2))\ehat(2\lambda,r,\theta_0) \,dr = 0.
$$
Since $\gamma$ was the curve $r \mapsto (r,\theta)$, this shows the result.
\end{proof}

\begin{proof}[Proof of Theorem \ref{theorem_main_local_result}]
Suppose that the local ray transform is injective on $O$ and $O \cap \partial M \subset E$. By Theorem \ref{theorem_main_local_integrals}, we know that 
\begin{equation} \label{cqdifference_integrals}
\int_0^L e^{-2\lambda t} (c(q_1-q_2))\ehat(2\lambda,\gamma(t)) \,dt = 0
\end{equation}
for any nontangential geodesic $\gamma$ in $O$. Setting $\lambda = 0$ and using local injectivity of the ray transform, we obtain that 
$$
(c(q_1-q_2))\ehat(0,\,\cdot\,) = 0 \text{ in } O.
$$
Going back to \eqref{cqdifference_integrals} and differentiating this identity with respect to $\lambda$, and then setting $\lambda = 0$ and using the vanishing of $(c(q_1-q_2))\ehat(0,\,\cdot\,)$ on $O$, it follows that 
$$
\int_0^L \frac{\partial}{\partial \lambda} \left[ (c(q_1-q_2))\ehat \right] (0,\gamma(t)) \,dt = 0 \text{ in } O
$$
for any nontangential geodesic in $O$. Local uniqueness for the ray transform again implies that 
$$
\frac{\partial}{\partial \lambda} \left[ (c(q_1-q_2))\ehat \right] (0,\,\cdot\,)= 0 \text{ in } O.
$$
Iterating this argument by taking higher order derivatives of \eqref{cqdifference_integrals} shows that 
$$
\left(\frac{\partial}{\partial \lambda}\right)^k \left[ (c(q_1-q_2))\ehat \right] (0,\,\cdot\,) = 0 \text{ in } O
$$
for any $k$. Since $c(q_1-q_2)$ is compactly supported in $x_1$, its Fourier transform is analytic and we have 
$$
(c(q_1-q_2))\ehat(\lambda,\,\cdot\,) = 0 \text{ in } O \text{ for all } \lambda \in \mR.
$$
Inverting the Fourier transform and using that $c$ is positive, we obtain that $q_1 = q_2$ in $M \cap (\mR \times O)$.
\end{proof}

\begin{proof}[Proof of Theorem \ref{theorem_main_global}]
Since $(M,g)$ is admissible, we may assume that 
$$
(M,g) \subset (\mR \times M_0, g), \quad g = c(e \oplus g_0)
$$
where $(M_0,g_0)$ is simple. 
The argument is very similar to the proof of Theorem \ref{theorem_main_local_integrals}, and we only indicate the required changes. Up to the formula \eqref{roneprime_equation}, the only change is that there is no restriction on $b \in C^{\infty}(S^{n-2})$ (we do not require $v_s$ to vanish on any part of the boundary). The function $r_1'$ is obtained as a solution of \eqref{roneprime_equation}, but this time we want that $\supp(\tilde{u}_1|_{\partial M}) \subset \partial M_-$. Fix $\delta > 0$. The boundary condition for $\tilde{u}_1$ is \eqref{roneprime_boundarycondition}, where $f_-$ is chosen to be 
$$
f_- = -e^{-i\lambda x_1} v_s(x') \quad \text{on } S_- \cup S_0.
$$

We use Proposition \ref{prop_carleman_solvability} to solve for $r_1'$. We have $\norm{f}_{L^2(M)} = O(1)$, and the bound $\norm{f_-}_{L^{\infty}} \lesssim 1$ implies 
$$
\norm{f_-}_{L^2(S_-)} \lesssim \sigma( \{ \partial_{\nu} x_1 \geq \delta \})
$$
and 
$$
\norm{f_-}_{L^2(S_0)} \lesssim \sigma( \{ -1/(3\tau) < \partial_{\nu} x_1 < 0 \}) + \sigma(\partial M_{tan}) + \sigma(\{ 0 < \partial_{\nu} x_1 < \delta \}).
$$
Now we use that 
$$
\sigma(\partial M_{tan}) = 0.
$$
This shows that we obtain the same estimate for $r_1'$ as before:
$$
\norm{r_1'}_{L^2(M)} \leq C_0 \left[ \tau^{-1} + (\delta \tau)^{-1/2} + o_{\tau \to \infty}(1) + o_{\delta \to 0}(1) \right].
$$
We can now continue as in the proof of Theorem \ref{theorem_main_local_integrals} to conclude that 
$$
\int_0^L e^{-2\lambda t} (c(q_1-q_2))\ehat(2\lambda,\gamma(t)) \,dt = 0
$$
for any $\lambda \in \mR$ and for any nontangential geodesic in $(M_0,g_0)$. The geodesic ray transform (with zero attenuation) is injective in $(M_0,g_0)$ \cite{S}. Following the proof of Theorem \ref{theorem_main_local_result}, but now using all the nontangential geodesics in $(M_0,g_0)$, shows that $q_1 = q_2$ in $M$.
\end{proof}

\section{Quasimodes concentrating near broken rays} \label{sec_gaussian}

In this section, to simplify notation, we write $(M,g)$ instead of $(M_0,g_0)$ and we assume that $(M,g)$ is a compact oriented Riemannian manifold having smooth boundary and $\dim(M) = m \geq 2$. Suppose that $E$ is a nonempty open subset of $\partial M$, and let $R = \partial M \setminus E$. We think of $E$ as the observation set where geodesics can enter and exit, and $R$ is the reflecting set. In the Calder\'on problem with partial data we are led to consider attenuated broken ray transforms, where one integrates a function on $M$ over broken geodesic rays that enter $M$ at some point of $E$, reflect nontangentially at points of $R$, and then exit $M$ at some point of $E$. The reflections will obey the law of geometric optics, so that a geodesic hitting the boundary in direction $v$ will be continued by the geodesic in the reflected direction $\hat{v} = v - 2\langle v, \nu \rangle \nu$.

Given a slightly complex frequency $s = \tau + i\lambda$, we will construct corresponding quasimodes, or approximate eigenfunctions, that concentrate near a fixed nontangential broken ray.

\begin{prop} \label{prop_gaussianbeam_quasimode_reflected}
Let $\gamma:[0,L] \to M$ be a nontangential broken ray with endpoints on $E$, and let $\lambda$ be a fixed real number. For any $K > 0$ there is a family $\{v_s \,;\,  s=\tau+i\lambda, \tau \geq 1\}$ in $C^{\infty}(M)$ such that as $\tau \to \infty$ 
$$
\norm{(-\Delta_{g} - s^2) v_s}_{L^2(M)} = O(\tau^{-K}), \quad \norm{v_s}_{L^2(M)} = O(1),
$$
the boundary values of $v_s$ satisfy 
$$
\norm{v_s}_{L^2(R)} = O(\tau^{-K}), \quad \norm{v_s}_{L^2(\partial M)} = O(1),
$$
and for any $\psi \in C(M)$ 
$$
\int_{M} \abs{v_{\tau+i\lambda}}^2 \psi \,dV_{g} \to \int_0^L e^{-2\lambda t} \psi(\gamma(t)) \,dt \quad \text{as $\tau \to \infty$}.
$$
\end{prop}

Let us begin by proving this result in the special case $E = \partial M$, so that $R = \emptyset$ and one does not need to worry about reflected rays. The next three preparatory lemmas describe a modified Fermi coordinate system that is very useful in this construction.

\begin{lemma} \label{lemma_selfintersection_elementary}
Let $(\hat{M},g)$ be a compact manifold without boundary, and let $\gamma: (a,b) \to \hat{M}$ be a unit speed geodesic segment that has no loops. There are only finitely many times $t \in (a,b)$ such that $\gamma$ intersects itself at $\gamma(t)$.
\end{lemma}
\begin{proof}
Since $\gamma$ has no loops, $(\gamma(t),\dot{\gamma}(t)) = (\gamma(t'), \dot{\gamma}(t'))$ implies $t = t'$. The first observation is that $\gamma$ can only self-intersect transversally, since also $(\gamma(t),\dot{\gamma}(t)) = (\gamma(t'), -\dot{\gamma}(t'))$ implies $t = t'$ (if this would happen for $t < t'$, then by uniqueness of geodesics $\dot{\gamma}(\frac{t+t'}{2}) = -\dot{\gamma}(\frac{t+t'}{2})$ which is impossible). Next note that if $r$ is smaller than the injectivity radius of $(\hat{M},g)$, then any two geodesic segments of length $\leq r$ can intersect transversally in at most one point (locally geodesics are close to straight lines). Partitioning $(a,b)$ in disjoint intervals $\{ J_k \}_{k=1}^K$ of length $\leq r$, we have an injective map 
\begin{multline*}
\{ (t,t') \in (a,b)^2 \,;\, t < t' \text{ and } \gamma(t) = \gamma(t') \} \\
 \mapsto \{ (k,l) \in \{1,\ldots,K\}^2 \,;\, t \in J_k, t' \in J_l \}.
\end{multline*}
Consequently, $\gamma$ can only self-intersect finitely many times.
\end{proof}

\begin{lemma} \label{lemma_diffeomorphism_elementary}
Let $F$ be a $C^1$ map from a neighborhood of $(a,b) \times \{0\}$ in $\mR^n$ into a smooth manifold such that $F|_{(a,b) \times \{0\}}$ is injective and $DF(t,0)$ is invertible for $t \in (a,b)$. If $[a_0,b_0]$ is a closed subinterval of $(a,b)$, then $F$ is a $C^1$ diffeomorphism in some neighborhood of $[a_0,b_0] \times \{0\}$ in $\mR^n$.
\end{lemma}
\begin{proof}
For any $t \in [a_0,b_0]$, the inverse function theorem implies that there is $\eps_t > 0$ such that $F|_{(t-3\eps_t,t+3\eps_t) \times B_{3\eps_t}(0)}$ is a $C^1$ diffeomorphism. Since $[a_0,b_0]$ is covered by the intervals $(t-\eps_t,t+\eps_t)$, by compactness we have $[a_0,b_0] \subset \cup_{j=1}^N (t_j-\eps_j,t_j+\eps_j)$ where $F|_{(t_j-3\eps_j,t_j+3\eps_j) \times B_{3\eps_j}(0)}$ is bijective. We can further assume (upon throwing away or shrinking some intervals if necessary) that the intervals $I_j = (t_j-\eps_j,t_j+\eps_j)$ satisfy $\overline{I}_j \cap \overline{I}_k = \emptyset$ unless $\abs{j-k} \leq 1$. Since $\gamma(t) = F(t,0)$ is injective, we also have $\gamma(\overline{I}_j) \cap \gamma(\overline{I}_k) = \emptyset$ unless $\abs{j-k} \leq 1$.

Fix a Riemannian metric in the target manifold, and define 
$$
\delta = \inf \ \{ \text{dist}(\gamma(\overline{I}_j), \gamma(\overline{I}_k)) \,;\, \abs{j-k} \geq 2 \} > 0.
$$
Let $U_j = I_j \times B_{\eps}(0)$, where $\eps < \min \{ \eps_1, \ldots, \eps_N \}$ is chosen so small that $F(U_j) \subset \{ q \,;\, \text{dist}(q,\gamma(\overline{I}_j)) < \delta/2 \}$. Then $F(U_j) \cap F(U_k) = \emptyset$ unless $\abs{j-k} \leq 1$. Define 
$$
U = \cup_{j=1}^N U_j.
$$
To show that $F|_U$ is a $C^1$ diffeomorphism, it is enough to check injectivity. If $F(t,y) = F(t',y')$ for $(t,y), (t',y') \in U$, then necessarily $(t,y) \in U_j$, $(t',y') \in U_k$ where $\abs{j-k} \leq 1$. We may assume that $\eps_j \geq \eps_k$. Since $F|_{(t_j-3\eps_j,t_j+3\eps_j) \times B_{3\eps_j}(0)}$ is bijective, we obtain $(t,y) = (t',y')$.
\end{proof}

\begin{lemma} \label{lemma_quasimode_coordinates}
Let $(\hat{M},g)$ be a compact manifold without boundary, and assume that $\gamma: (a,b) \to \hat{M}$ is a unit speed geodesic segment with no loops. Given a closed subinterval $[a_0,b_0]$ of $(a,b)$ such that $\gamma|_{[a_0,b_0]}$ self-intersects only at times $t_j$ with $a_0 < t_1 < \ldots < t_N < b_0$ (set $t_0 = a_0$ and $t_{N+1} = b_0$), there is an open cover $\{(U_j,\varphi_j)\}_{j=0}^{N+1}$ of $\gamma([a_0,b_0])$ consisting of coordinate neighborhoods having the following properties:
\begin{enumerate}
\item 
$\varphi_j(U_j) = I_j \times B$ where $I_j$ are open intervals and $B = B(0,\delta)$ is an open ball in $\mR^{n-1}$ where $\delta$ can be taken arbitrarily small,
\item 
$\varphi_j(\gamma(t)) = (t,0)$ for $t \in I_j$,
\item 
$t_j$ only belongs to $I_j$ and $\overline{I}_j \cap \overline{I}_k = \emptyset$ unless $\abs{j-k} \leq 1$,
\item 
$\varphi_j = \varphi_k$ on $\varphi_j^{-1}((I_j \cap I_k) \times B)$.
\end{enumerate}
Further, if $S$ is a hypersurface through $\gamma(a_0)$ that is transversal to $\dot{\gamma}(a_0)$, one can arrange that the map $y \mapsto \varphi_0^{-1}(a_0,y)$ parametrizes $S$ near $\gamma(a_0)$.
\end{lemma}
\begin{proof}
We will use modified Fermi coordinates, constructed as follows. Let $\{ v_1, \ldots, v_{n-1} \}$ be an orthonormal set of vectors in $T_{\gamma(a_0)} \hat{M}$ such that $\{ \dot{\gamma}(a_0), v_1, \ldots, v_{n-1} \}$ is a basis. (The case where $\{ \dot{\gamma}(a_0), v_1, \ldots, v_{n-1} \}$ is an orthonormal basis corresponds to the usual Fermi coordinates.) Let $E_{\alpha}(t)$ be the parallel transport of $v_{\alpha}$ along the geodesic $\gamma$. Since $\dot{\gamma}(t)$ is also parallel along $\gamma$, the set $\{ \dot{\gamma}(t), E_1(t), \ldots, E_{n-1}(t) \}$ is a basis of $T_{\gamma(t)} \hat{M}$ for $t \in (a,b)$.

Define the function 
$$
F: (a,b) \times \mR^{n-1} \to \hat{M}, \ \ F(t,y) = \exp_{\gamma(t)}(y^{\alpha} E_{\alpha}(t)).
$$
Here $\exp$ is the exponential map in $(\hat{M},g)$ and $\alpha$, $\beta$ run from $1$ to $n-1$. Then $F(t,0) = \gamma(t)$ and (with $e_{\alpha}$ the $\alpha$th coordinate vector) 
$$
\frac{\partial}{\partial s} F(t,s e_{\alpha})\big|_{s=0} = E_{\alpha}(t), \quad \frac{\partial}{\partial t} F(t,0) = \dot{\gamma}(t).
$$
Thus $F$ is a $C^{\infty}$ map near $(a,b) \times \{0\}$ such that $DF(t,0)$ is invertible for $t \in (a,b)$.

In the case where $\gamma$ does not self-intersect, $F|_{(a,b) \times \{0\}}$ is injective and Lemma \ref{lemma_diffeomorphism_elementary} implies the existence of a single coordinate neighborhood of $\gamma([a_0,b_0])$ so that (1) and (2) are satisfied (then (3) and (4) are void). In the general case, by Lemma \ref{lemma_selfintersection_elementary} the geodesic segment $\gamma|_{[a_0,b_0]}$ only self-intersects at finitely many times $t_j$ with $a_0 < t_1 < \ldots < t_N < b_0$. For some sufficiently small $\delta$, $\gamma$ is injective on the intervals $(a,t_1-\delta)$, $(t_1-2\delta,t_2-\delta)$, \ldots, $(t_N-2\delta,b)$ and each interval intersects at most two of the others. Restricting the map $F$ above to suitable neighborhoods corresponding to these intervals (or slightly smaller ones) and using Lemma \ref{lemma_diffeomorphism_elementary}, we obtain the required coordinate charts with $\varphi_j = F^{-1}|_{U_j}.$

Let $S$ be a hypersurface transversal to $\dot{\gamma}(a_0)$, and choose some para\-metrization $y \mapsto q(y)$ of $S$ near $\gamma(a_0)$ satisfying $\frac{\partial}{\partial s} q(s e_{\alpha}) = v_{\alpha}$. We will form a new chart $(\tilde{U}_0,\tilde{\varphi}_0)$ by modifying $(U_0,\varphi_0)$ so that $y \mapsto \tilde{\varphi}_0^{-1}(a_0,y)$ parametrizes $S$ near $\gamma(a_0)$.

We may assume that $a_0 = 0$, and write $F_0 = \varphi_0^{-1}$, $\tilde{F}_0 = \tilde{\varphi}_0^{-1}$. It is enough to choose $\tilde{F}_0 = F_0 \circ \Phi$, where $\Phi$ is a diffeomorphism near $\overline{I}_0 \times B$ such that 
\begin{gather*}
\Phi(t,0) = (t,0), \\
\Phi(0,y) = F_0^{-1}(q(y)), \\
\Phi(t,y) = (t,y) \ \ \text{for $t > c$ with suitable $c > 0$}.
\end{gather*}
Write the components of $\tilde{q} = F_0^{-1} \circ q$ as Taylor series 
$$
\tilde{q}^j(y) = \tilde{q}^j(0) + \nabla \tilde{q}^j(0) \cdot y + H^j(y)y \cdot y
$$
where $H^j$ are smooth matrices, and $j = 0, \ldots, n-1$ ($t$ is the $0$th variable). The properties of $q$ imply that 
$$
\tilde{q}^j(0) = 0, \ \ \partial_{\beta} \tilde{q}^0(0) = 0, \ \ \partial_{\beta} \tilde{q}^{\alpha}(0) = \delta_{\beta}^{\alpha}.
$$
We look for $\Phi$ in the form 
$$
\Phi^j(t,y) = f^j(t) + a^j(t) \cdot y + R^j(t,y) y \cdot y
$$
for some smooth functions $f^j$, vectors $a^j$ and matrices $R^j$. The conditions for $\Phi$ motivate the following choices:
$$
f^0(t) = t, \ \ f^{\alpha}(t) = 0, \ \ a^0_{\beta}(t) = 0, \ \ a^{\alpha}_{\beta}(t) = \delta^{\alpha}_{\beta}.
$$
We choose $R^j(t,y)$ to be a smooth matrix with $R^j(0,y) = H^j(y)$ and $R^j(t,y) = 0$ for $t > c$. Then $D\Phi(t,0) = \mathrm{Id}$, and Lemma \ref{lemma_diffeomorphism_elementary} ensures that $\Phi$ is a diffeomorphism near $\overline{I}_0 \times B$ after possible decreasing $B$.
\end{proof}

The next result gives the construction of (non-reflected) Gaussian beam quasimodes associated with a finite length geodesic segment that enters and exits the domain nontangentially. To prepare for the reflected case, we also consider the possibility of prescribing the boundary values of the quasimode at least up to high order at a point. Recall that if $f$ is a smooth function having a critical point at $p$, the Hessian of $f$ at $p$ is the quadratic form 
$$
\text{Hess}_p(f)(\dot{\eta}(0),\dot{\eta}(0)) = (f \circ \eta)''(0)
$$
where $\eta$ is any smooth curve with $\eta(0) = p$.

\begin{prop} \label{prop_gaussianbeam_quasimode}
Let $\gamma:[0,L] \to M$ be any unit speed geodesic in $(M,g)$ such that $\gamma(0), \gamma(L) \in \partial M$, $\dot{\gamma}(0)$ and $\dot{\gamma}(L)$ are nontangential, and $\gamma(t) \in M^{\text{int}}$ for $0 < t < L$. Let also $\lambda$ be a fixed real number. For any $K > 0$ there is a family $\{v_s \,;\,  s=\tau+i\lambda, \tau \geq 1\}$ in $C^{\infty}(M)$ such that as $\tau \to \infty$ 
\begin{gather*}
\norm{(-\Delta_{g} - s^2) v_s}_{L^2(M)} = O(\tau^{-K}), \quad \norm{v_s}_{L^2(M)} = O(1), \\
\norm{v_s}_{L^2(\partial M)} = O(1),
\end{gather*}
and for any $\psi \in C(M)$ 
\begin{equation} \label{vs_limit_measure}
\int_{M} \abs{v_{\tau+i\lambda}}^2 \psi \,dV_{g} \to \int_0^L e^{-2\lambda t} \psi(\gamma(t)) \,dt \quad \text{as $\tau \to \infty$}.
\end{equation}
Given any neighborhood of $\gamma([0,L])$ one can arrange that each $v_s$ is supported in this neighborhood, and away from the points where $\gamma$ self-intersects one has $v_s = e^{is\Theta} a$ where $\Theta$ and $a$ are smooth complex functions with 
$$
d\Theta(\gamma(t)) = \dot{\gamma}(t)^{\flat}, \quad a(\gamma(t)) \neq 0 \text{ for $\tau$ large}.
$$
If $\gamma$ does not self-intersect at $\gamma(0)$, the $K$th order jets of $\Theta|_{\partial M}$ and $a|_{\partial M}$ can be prescribed freely at $\gamma(0)$ except for the following restrictions: $d\Theta(\gamma(0)) = \dot{\gamma}(0)^{\flat}$, the Hessian of $\im(\Theta|_{\partial M})$ at $\gamma(0)$ is positive definite, and $a(\gamma(0)) \neq 0$.
\end{prop}
\begin{proof}
We embed $(M,g)$ in a compact manifold $(\hat{M},g)$ without boundary and extend $\gamma$ as a unit speed geodesic in $\hat{M}$. Choose $\eps > 0$ so that $\gamma(t)$ lies in $\hat{M} \smallsetminus M$ and does not self-intersect for $t \in [-2\eps,0) \cup (L,L+2\eps]$. We will construct a Gaussian beam quasimode in a neighborhood of $\gamma([-\eps,L+\eps])$.

Fix a point $p_0 = \gamma(t_0)$ on $\gamma([-\eps,L+\eps])$ and let $(t,y)$ be any local coordinates near $p_0$, defined in $U = \{ (t,y) \,;\, t \in I, \abs{y} < \delta \}$ for some open interval $I$ containing $t_0$, such that $p_0$ corresponds to $(t_0,0)$ and the geodesic near $p_0$ is given by $\Gamma = \{ (t,0) \,;\, t \in I \}$. Write $x = (t,y)$ where $x_1 = t$ and $(x_2,\ldots,x_m) = y$. We seek to find a quasimode $v_s$ concentrated near $\Gamma$, having the form 
$$
v_s = e^{is\Theta} a
$$
where $s = \tau+i\lambda$, and $\Theta$ and $a$ are smooth complex functions near $\Gamma$ with $a$ supported in $\{ \abs{y} < \delta/2 \}$.

We compute 
$$
(-\Delta - s^2) v_s = f
$$
where 
$$
f = e^{is\Theta}(s^2 [(\langle d\Theta, d\Theta \rangle - 1)a] - is[2\langle d\Theta, da \rangle + (\Delta \Theta) a] - \Delta a).
$$
We first choose $\Theta$ so that 
\begin{equation} \label{theta_nth_order}
\langle d\Theta, d\Theta \rangle = 1 \quad \text{to $N$th order on $\Gamma$}.
\end{equation}
In fact we look for $\Theta$ of the form $\Theta = \sum_{j=0}^N \Theta_j$ where 
$$
\Theta_j(t,y) = \sum_{\abs{\gamma} = j} \frac{\Theta_{j,\gamma}(t)}{\gamma!} y^{\gamma}.
$$
We also write $g^{jk} = \sum_{l=0}^N g^{jk}_l + g^{jk}_{N+1}$ where 
$$
g^{jk}_l(t,y) = \sum_{\abs{\beta} = l} \frac{g^{jk}_{l,\beta}(t)}{\beta!} y^{\beta}, \quad g^{jk}_{N+1} = O(\abs{y}^{N+1}).
$$
Set $g^{jk}_l = 0$ for $l \geq N+2$.

With the understanding that $j,k$ run from $1$ to $m$ and $\alpha,\beta$ run from $2$ to $m$, the main part of the argument will consist of finding suitable $\Theta_0$, $\Theta_1$ and $\Theta_2$ in the following form:
\begin{gather*}
\Theta_0(t) \text{ real valued}, \\
\Theta_1(t) = \xi_{\alpha}(t) y^{\alpha} \text{ with $\xi_{\alpha}(t)$ real valued}, \\
\Theta_2(t) = \frac{1}{2} H_{\alpha \beta}(t) y^{\alpha} y^{\beta}
\end{gather*}
where $H(t) = (H_{\alpha \beta}(t))$ is a complex symmetric matrix, $H_{\alpha \beta} = H_{\beta \alpha}$, such that $\im(H(t))$ is positive definite for all $t$. We also write 
$$
\xi_1(t) = \partial_t \Theta_0(t).
$$

Since $\partial_t \Theta_0 = \xi_1$ and $\partial_{\alpha} \Theta_1 = \xi_{\alpha}$, we compute 
\begin{multline*}
g^{jk} \partial_j \Theta \partial_k \Theta - 1 = g^{11}(\partial_t \Theta_0+\partial_t \Theta_1+\ldots)(\partial_t \Theta_0+\partial_t \Theta_1+\ldots) \\
 + 2 g^{1\alpha} (\partial_t \Theta_0+\partial_t \Theta_1+\ldots)(\partial_{\alpha} \Theta_1+\partial_{\alpha} \Theta_2+\ldots) \\
 + g^{\alpha \beta}(\partial_{\alpha} \Theta_1+\partial_{\alpha} \Theta_2+\ldots)(\partial_{\beta} \Theta_1+\partial_{\beta} \Theta_2+\ldots) - 1 \\
  = g^{jk} \xi_j \xi_k + 2 g^{11} \xi_1 (\partial_t \Theta_1 + \ldots) + 2 g^{1 \alpha} \xi_1 (\partial_{\alpha} \Theta_2 + \ldots) + 2 g^{1 \alpha} \xi_{\alpha} (\partial_t \Theta_1 + \ldots) \\
  + 2 g^{\alpha \beta} \xi_{\alpha} (\partial_{\beta} \Theta_2 + \ldots) + g^{11}(\partial_t \Theta_1+\ldots)(\partial_t \Theta_1+\ldots) \\
  + 2 g^{1\alpha} (\partial_t \Theta_1+\ldots)(\partial_{\alpha} \Theta_2+\ldots) + g^{\alpha \beta}(\partial_{\alpha} \Theta_2+\ldots)(\partial_{\beta} \Theta_2+\ldots) - 1 \\
  = g^{jk} \xi_j \xi_k + 2 g^{1k} \xi_k (\partial_t \Theta_1 + \ldots) + 2 g^{\alpha k} \xi_k (\partial_{\alpha} \Theta_2 + \ldots) + g^{11}(\partial_t \Theta_1+\ldots)(\partial_t \Theta_1+\ldots) \\
  + 2 g^{1\alpha} (\partial_t \Theta_1+\ldots)(\partial_{\alpha} \Theta_2+\ldots) + g^{\alpha \beta}(\partial_{\alpha} \Theta_2+\ldots)(\partial_{\beta} \Theta_2+\ldots) - 1.
\end{multline*}
Writing $g^{jk} = g^{jk}_0 + g^{jk}_1 + \ldots$ and grouping like powers of $y$, we obtain 
\begin{multline} \label{thetaequation_firstref}
g^{jk} \partial_j \Theta \partial_k \Theta - 1 = [ g^{jk}_0 \xi_j \xi_k - 1] \\
 + [g_1^{jk} \xi_j \xi_k +  2 g_0^{1k} \xi_k \dot{\xi}_{\beta} y^{\beta} + 2 g_0^{\alpha k} \xi_k H_{\alpha \beta} y^{\beta} ] \\
 + (g_2^{jk} + \ldots) \xi_j \xi_k + 2 g_0^{1k} \xi_k (\partial_t \Theta_2 + \ldots) + 2 (g_1^{1k} + \ldots) \xi_k (\partial_t \Theta_1 + \ldots) \\
 + 2 g_0^{\alpha k} \xi_k (\partial_{\alpha} \Theta_3 + \ldots) + 2 (g_1^{\alpha k} + \ldots) \xi_k (\partial_{\alpha} \Theta_2 + \ldots) + g^{11}(\partial_t \Theta_1+\ldots)(\partial_t \Theta_1+\ldots) \\
  + 2 g^{1\alpha} (\partial_t \Theta_1+\ldots)(\partial_{\alpha} \Theta_2+\ldots) + g^{\alpha \beta}(\partial_{\alpha} \Theta_2+\ldots)(\partial_{\beta} \Theta_2+\ldots).
\end{multline}

We can make the two expressions in brackets to vanish by choosing $\xi(t)$ to be part of the solution $(x(t),\xi(t))$ of the cogeodesic flow with Hamiltonian $h(x,\xi) = \frac{1}{2} g^{jk}(x) \xi_j \xi_k$, 
\begin{align*}
\dot{x}^j(t) &= \partial_{\xi_j} h(x(t),\xi(t)), \\
\dot{\xi}_j(t) &= -\partial_{x_j} h(x(t),\xi(t)).
\end{align*}
There is a unique solution with $x(t_0) = p_0$ and $\xi(t_0) = \dot{\gamma}(t_0)^{\flat}$ (here we raise and lower indices with respect to the metric $g$). It follows that $x(t)$ is the unit speed geodesic $t \mapsto (t,0)$, and $\xi^j(t) = \dot{x}^j(t)$. Then $g_0^{jk} \xi_j \xi_k = 1$, and with our choice of coordinates $\xi^1 = 1$ and $\xi^{\alpha} = 0$ so that also 
$$
g_0^{1k} \xi_k = 1, \quad g_0^{\alpha k} \xi_k = 0.
$$
We further have 
$$
\dot{\xi}_{\beta} y^{\beta} = -\frac{1}{2} \partial_{x_{\beta}} g^{jk}(t,0) \xi_j \xi_k y^{\beta} = -\frac{1}{2} g_1^{jk} \xi_j \xi_k.
$$

Noting that $\partial_1$ has unit length, we have 
$$
\xi_1 = g_{1k}(t,0) \xi^k = 1.
$$
Since $\xi_{\alpha} = g_{\alpha k}(t,0) \xi^k =g_{\alpha 1}(t,0)$, we can therefore choose 
\begin{gather*}
\Theta_0(t) = t, \\
\Theta_1(t) = g_{\alpha 1}(t,0) y^{\alpha}.
\end{gather*}
Using these choices and the facts above, in \eqref{thetaequation_firstref} the expressions in brackets will indeed vanish and one obtains 
\begin{multline*}
g^{jk} \partial_j \Theta \partial_k \Theta - 1 =  (g_2^{jk} + \ldots) \xi_j \xi_k + 2 (\partial_t \Theta_2 + \ldots) + 2 (g_1^{1k} + \ldots) \xi_k (\partial_t \Theta_1 + \ldots) \\
 + 2 (g_1^{\alpha k} + \ldots) \xi_k (\partial_{\alpha} \Theta_2 + \ldots) + g^{11}(\partial_t \Theta_1+\ldots)(\partial_t \Theta_1+\ldots) \\
  + 2 g^{1\alpha} (\partial_t \Theta_1+\ldots)(\partial_{\alpha} \Theta_2+\ldots) + g^{\alpha \beta}(\partial_{\alpha} \Theta_2+\ldots)(\partial_{\beta} \Theta_2+\ldots) \\
  = \Big[ g_2^{jk} \xi_j \xi_k + 2 \partial_t \Theta_2 + 2 g_1^{\alpha k} \xi_k \partial_{\alpha} \Theta_2 + 2 g_0^{1 \alpha} \partial_t \Theta_1 \partial_{\alpha} \Theta_2 + g_0^{\alpha \beta} \partial_{\alpha} \Theta_2 \partial_{\beta} \Theta_2 \\
  + 2 g_1^{1k} \xi_k \partial_t \Theta_1 + g_0^{11} (\partial_t \Theta_1)^2 \Big] \\
  + \sum_{p=3}^N \Big[ g_p^{jk} \xi_j \xi_k + 2 \partial_t \Theta_p + 2 g_1^{\alpha k} \xi_k \partial_{\alpha} \Theta_p + 2 g_0^{1 \alpha} \partial_t \Theta_1 \partial_{\alpha} \Theta_p + 2 g_0^{\alpha \beta} \partial_{\alpha} \Theta_2 \partial_{\beta} \Theta_p \\
  + 2 \sum_{l=1}^{p-1} g_{p-l}^{1k} \xi_k \partial_t \Theta_l 
  + 2 \sum_{l=2}^{p-1} g_{p-l+1}^{\alpha k} \xi_k \partial_{\alpha} \Theta_l + \sum_{l=0}^{p-2} g_l^{11} \sum_{\overset{r+s=p-l}{1 \leq r,s < p}} \partial_t \Theta_r \partial_t \Theta_s \\
  + \sum_{l=0}^{p-2} g_l^{1 \alpha} \sum_{\overset{r+s=p-l+1}{\overset{1 \leq r < p}{2 \leq s < p}}} \partial_t \Theta_r \partial_{\alpha} \Theta_s + \sum_{l=0}^{p-2} g_l^{\alpha \beta} \sum_{\overset{r+s=p-l+2}{2 \leq r,s < p}} \partial_{\alpha} \Theta_r \partial_{\beta} \Theta_s \Big] + O(\abs{y}^{N+1}).
\end{multline*}

We want to choose $\Theta_2$ so that the first term in brackets vanishes. Recalling that we are looking for $\Theta_2$ in the form $\Theta_2(t,y) = \frac{1}{2} H_{\alpha \beta}(t) y^{\alpha} y^{\beta}$ where $H(t)$ is a smooth complex symmetric matrix, it follows that $H$ should satisfy the matrix equation 
\begin{equation*}
\dot{H}_{\alpha \beta} y^{\alpha} y^{\beta} + 2 g_1^{\gamma k} \xi_k H_{\gamma \beta} y^{\beta} + 2 g_0^{1 \gamma} \partial_t \Theta_1 H_{\gamma \beta} y^{\beta} + g_0^{\gamma \delta} H_{\gamma \alpha} H_{\delta \beta} y^{\alpha} y^{\beta} = F_{\alpha \beta} y^{\alpha} y^{\beta}
\end{equation*}
where $F(t)$ is a real valued smooth symmetric matrix. This can be further written as the matrix Riccati equation 
$$
\dot{H} + BH + HB^t + HCH = F
$$
where $B(t)$ and $C(t)$ are real smooth matrices and $C$ is symmetric. More precisely, since $g_1^{jk} = \partial_{\alpha} g^{jk}(t,0) y^{\alpha}$ we have 
\begin{equation} \label{gaussianbeam_b_c_definition}
B_{\alpha}^{\gamma} = \partial_{\alpha} g^{\gamma k}(t,0) \xi_k + g_0^{1 \gamma} \dot{\xi}_{\alpha}, \quad C^{\gamma \delta} = g_0^{\gamma \delta}.
\end{equation}
Choosing $H(t_0) = H_0$ where $H_0$ is a complex symmetric matrix with $\im(H_0)$ positive definite, it follows that the Riccati equation has a unique smooth complex symmetric solution $H(t)$ with $\im(H(t))$ positive definite \cite{KKL}. This completes the construction of $\Theta_2$. From the lower order terms we can find $\Theta_3,\ldots,\Theta_N$ successively by solving linear first order ODEs on $\Gamma$ with prescribed initial conditions at $t_0$. In this way, we obtain a smooth $\Theta$ satisfying \eqref{theta_nth_order}.

The next step is to find $a$ such that 
$$
s[2\langle d\Theta, da \rangle + (\Delta \Theta) a] - i\Delta a = 0 \quad \text{to $N$th order on $\Gamma$}.
$$
We look for $a$ in the form 
$$
a = \tau^{\frac{m-1}{4}}(a_0 + s^{-1} a_{-1} + \ldots + s^{-N} a_{-N}) \chi(y/\delta')
$$
where $\chi$ is a smooth function with $\chi = 1$ for $\abs{y} \leq 1/4$ and $\chi = 0$ for $\abs{y} \geq 1/2$. Writing $\eta = \Delta \Theta$, it is sufficient to determine $a_j$ so that 
\begin{align*}
2\langle d\Theta,d a_0 \rangle + \eta a_0 &= 0 \quad \text{to $N$th order on $\Gamma$}, \\
2\langle d\Theta,d a_{-1} \rangle + \eta a_{-1} - i\Delta a_0 &= 0 \quad \text{to $N$th order on $\Gamma$}, \\
\vdots, & \\
2\langle d\Theta,d a_{-N} \rangle + \eta a_{-N} - i\Delta a_{-(N-1)} &= 0 \quad \text{to $N$th order on $\Gamma$}.
\end{align*}
Consider $a_0 = a_{00} + \ldots + a_{0N}$ where $a_{0j}(t,y)$ is a polynomial of order $j$ in $y$, and similarly let $\eta = \eta_0 + \ldots + \eta_N$. We compute 
\begin{multline*}
2\langle d\Theta,d a_0 \rangle + \eta a_0 = 2(g^{11}_0+\ldots)(\partial_t \Theta_0+\ldots)(\partial_t a_{00}+\ldots) \\
 + 2(g^{1\alpha}_0+\ldots)(\partial_t \Theta_0+\ldots)(\partial_{\alpha} a_{01}+\ldots) \\
 + 2(g^{1\alpha}_0+\ldots)(\partial_{\alpha} \Theta_1+\ldots)(\partial_t a_{00}+\ldots) \\
 + 2(g_0^{\alpha \beta}+\ldots)(\partial_{\beta} \Theta_1+\ldots)(\partial_{\alpha} a_{01}+\ldots) + (\eta_0+\eta_1+\ldots)(a_{00}+a_{01}+\ldots).
\end{multline*}
Recalling that $\partial_t \Theta_0 = \xi_1 = 1$, $\partial_{\alpha} \Theta_1 = \xi_{\alpha}$ where $g_0^{1j} \xi_j = 1$ and $g_0^{\alpha j} \xi_j = 0$, we obtain 
\begin{multline*}
2\langle d\Theta,d a_0 \rangle + \eta a_0 = 2 \big[ g_0^{11} \xi_1 + g_0^{11} (\partial_t \Theta_1 + \ldots) + (g_1^{11}+\ldots)(\partial_t \Theta_0 + \ldots) \\
 + g_0^{1 \alpha} \xi_{\alpha} + g_0^{1 \alpha}(\partial_{\alpha} \Theta_2 + \ldots) + (g_1^{1 \alpha} + \ldots)(\partial_{\alpha} \Theta_1 + \ldots) \big] (\partial_t a_{00} + \ldots) \\
 + 2 \big[ g_0^{1 \alpha}(\partial_t \Theta_1 + \ldots) + (g_1^{1 \alpha} + \ldots)(\partial_t \Theta_0 + \ldots) + g_0^{\alpha \beta} (\partial_{\beta} \Theta_2 + \ldots) \\
 + (g_1^{\alpha \beta} + \ldots)(\partial_{\beta} \Theta_1 + \ldots) \big] (\partial_{\alpha} a_{01} + \ldots) + (\eta_0+\eta_1+\ldots)(a_{00}+a_{01}+\ldots) \\
 = [ 2 \partial_t a_{00} + \eta_0 a_{00} ] + \sum_{p=1}^N \Big[ 2 \partial_t a_{0p} + q_p^{\alpha \beta} y^{\beta} \partial_{\alpha} a_{0p} + \eta_0 a_{0p} + F_p \Big] + O(\abs{y}^{N+1})
\end{multline*}
where $q_p^{\alpha \beta}(t)$ are smooth functions only depending on $g$ and $\Theta$, and $F_p(t,y)$ is a polynomial of degree $p$ in $y$ that only depends on $g$, $\Theta$, $\eta$ and $a_{00}, \ldots, a_{0,p-1}$.

We want to choose $a_{00}$ so that the first term in brackets vanishes, that is, 
$$
\partial_t a_{00} + \frac{1}{2} \eta_0 a_{00} = 0.
$$
This has the solution 
$$
a_{00}(t) = c_0 e^{-\frac{1}{2} \int_{t_0}^t \eta_0(s) \,ds}, \quad a_{00}(t_0) = c_0.
$$
We obtain $a_{01},\ldots,a_{0N}$ successively by solving linear first order ODEs with prescribed initial conditions at $t_0$. The functions $a_1,\ldots,a_N$ may be determined in a similar way so that the required equations are satisfied to $N$th order on $\Gamma$. This completes the construction of $a$.

To review what has been achieved so far, we have constructed a function $v_s = e^{is\Theta} a$ in $U$ where 
\begin{align*}
\Theta(t,y) &= t + \xi_{\alpha}(t) y^{\alpha} + \frac{1}{2} H_{\alpha \beta}(t) y^{\alpha} y ^{\beta} + \tilde{\Theta}, \\
a(t,y) &= \tau^{\frac{m-1}{4}}(a_0 + s^{-1} a_{-1} + \ldots + s^{-N} a_{-N}) \chi(y/\delta'), \\
a_0(t,0) &= c_0 e^{-\frac{1}{2} \int_{t_0}^t \eta_0(s) \,ds}.
\end{align*}
Here $\tilde{\Theta} = O(\abs{y}^3)$ and $\tilde{\Theta}$ and each $a_j$ are independent of $\tau$. Also, $f = (-\Delta-s^2)v_s$ has the form 
$$
f = e^{is\Theta} \tau^{\frac{m-1}{4}}(s^2 h_2 a + s h_1 + \ldots + s^{-(N-1)} h_{-(N-1)} + i s^{-N} \Delta(a_{-N} \chi(y/\delta')))
$$
where for each $j$ one has $h_j = 0$ to $N$th order on $\Gamma$. We also note that $d\Theta(\gamma(t)) = \dot{\gamma}(t)^{\flat}$ and $\text{Hess}_{\gamma(t_0)}(\im(\Theta|_{\{t=t_0\}})) = \im(H(t_0))$.

To prove the norm estimates for $v_s$ in $U$, note that 
\begin{align*}
\abs{e^{is\Theta}} &= e^{-\lambda \re\,\Theta} e^{-\tau \im\,\Theta} = e^{-\lambda t} e^{-\frac{1}{2}\tau \im(H(t))y \cdot y} e^{-\lambda O(\abs{y})} e^{-\tau O(\abs{y}^3)}.
\end{align*}
Here $\im(H(t))y \cdot y \geq c\abs{y}^2$ for $(t,y) \in U$ where $c>0$ depends on $H_0$, $g$ and $I$. By decreasing $\delta'$ if necessary, this shows the following bound when $t$ in a fixed compact set:
$$
\abs{v_s(t,y)} \lesssim \tau^{\frac{m-1}{4}} e^{-\frac{1}{4}c \tau \abs{y}^2} \chi(y/\delta').
$$
Integrating the square of this over $U$ we get, as $\tau \to \infty$, 
$$
\norm{v_s}_{L^2(U)} \lesssim \norm{\tau^{\frac{m-1}{4}} e^{-\frac{1}{4}c \tau \abs{y}^2}}_{L^2(U)} = O(1).
$$
Similarly we have 
\begin{align*}
\norm{(-\Delta-s^2)v_s}_{L^2(U)} &\lesssim \norm{\tau^{\frac{m-1}{4}} e^{-\frac{1}{4}c \tau \abs{y}^2}(\tau^2 \abs{y}^{N+1} + \tau^{-N})}_{L^2(U)} \\
 &= O(\tau^{\frac{3-N}{2}}).
\end{align*}
The norm estimates for $v_s$ in $U$ follow upon replacing $N$ by $2K+3$.

For the $L^2(\partial M)$ estimate, if $U$ contains a boundary point $x_0 = (t_0,0) \in \partial M$, then by assumption $\frac{\partial}{\partial t}|_{x_0}$ is transversal to $\partial M$. If $\rho$ is a boundary defining function for $M$, so $\partial M$ is given as the zero set $\rho(t,y) = 0$ near $x_0$ and $\nabla \rho = -\nu$ on $\partial M$, then $\frac{\partial \rho}{\partial t}(x_0) \neq 0$ and by the implicit function theorem there is a smooth function $y \mapsto t(y)$ near $0$ such that $\partial M$ is given by $\{ (t(y), y) \,;\, \abs{y} < r_0 \}$ near $x_0$. The bound for $v_s$ given above implies that for $\delta'$ small 
\begin{align*}
\norm{v_s}_{L^2(\partial M \cap U)}^2 = \int_{\abs{y} < r_0} \abs{v_s(t(y),y)} \,dS(y) \\
 \lesssim \int_{\abs{y} < r_0} \tau^{\frac{m-1}{2}} e^{-\frac{1}{2}c \tau \abs{y}^2} \,dy = O(1).
\end{align*}

At this point we can construct the quasimode $v_s$ in $M$ from the corresponding quasimodes defined on small pieces. Let $\gamma([-\eps,L+\eps])$ be covered by open sets $U^{(0)}, \ldots, U^{(N+1)}$ as in Lemma \ref{lemma_quasimode_coordinates}, and note that each $U^{(j)}$ corresponds to $I_j \times B(0,\delta)$ in the $(t,y)$ coordinates. Suppose first that $\gamma$ does not self-intersect at time $t=0$. We find a quasimode $v^{(0)} = e^{is\Theta^{(0)}} a^{(0)}$ in $U^{(0)}$ by the above procedure, with some fixed initial conditions at $t=0$ for the ODEs determining $\Theta^{(0)}$ and $a^{(0)}$. Choose some $t_0'$ with $\gamma(t_0') \in U^{(0)} \cap U^{(1)}$, and construct a quasimode $v^{(1)} = e^{is\Theta^{(1)}} a^{(1)}$ in $U^{(1)}$ by choosing the initial conditions for the ODEs for $\Theta^{(1)}$ and $a^{(1)}$ at $t_0'$ to be the corresponding values of $\Theta^{(0)}$ and $a^{(0)}$ at $t_0'$. Continuing in this way we obtain $v^{(2)}, \ldots, v^{(N+1)}$. If $\gamma$ self-intersects at $t=0$, we start the construction from $v^{(1)}$ fixing initial conditions for the ODEs at $t=0$, and find $v^{(0)}$ by going backward.

Let $\{\tilde{\chi}_j(t)\}$ be a partition of unity near $[-\eps,L+\eps]$ corresponding to the cover $\{I_j\}$, let $\chi_j(t,y) = \tilde{\chi}_j(t)$ on $U^{(j)}$, and define 
$$
v_s = \sum_{j=0}^{N+1} \chi_j v^{(j)}.
$$
Then $v_s$ is smooth in $\hat{M}$ and it is supported in a small neighborhood of $\gamma([-\eps,L+\eps])$. The important point is that since the ODEs for the phase functions and amplitudes have the same initial data in $U^{(j)}$ and in $U^{(j+1)}$, and since the local coordinates $\varphi_j$ and $\varphi_{j+1}$ coincide on $\varphi_j^{-1}((I_j \cap I_{j+1}) \times B)$, one actually has $v^{(j)} = v^{(j+1)}$ in $\varphi_j^{-1}((I_j \cap I_{j+1}) \times B)$. Letting $p_1,\ldots, p_R$ be the points where $\gamma$ intersects itself, we choose an open cover of $\supp(v_s) \cap M$,
$$
\supp(v_s) \cap M \subset \left( \cup_{r=1}^R V_r \right) \cup \left( \cup_{j=0}^{N+1} (W_{j,0} \cup W_{j,1}) \right),
$$
where $V_r$ are small neighborhoods the points $p_r$ and $W_{j,0}, W_{j,1} \subset U^{(j)}$, such that 
\begin{gather*}
v_s|_{V_r} = \sum_{\gamma(t_j) = p_r} v^{(j)}, \\
v_s|_{W_{j,l}} = v^{(j+l)}.
\end{gather*}
Since $v_s$ is a finite sum of the $v^{(j)}$ in each case, the $L^2(M)$ bounds for $v_s$ and $(-\Delta-s^2) v_s$ and the $L^2(\partial M)$ bounds for $v_s$ follow from corresponding bounds for the $v^{(j)}$. The form of $v_s$ near points where $\gamma$ does not self-intersect and the possibility to prescribe the $K$th order jets of $\Theta|_{\partial M}$ and $a|_{\partial M}$ at $\gamma(0)$ follow from the construction and Lemma \ref{lemma_quasimode_coordinates}.

To conclude the proof, using a partition of unity it is enough to verify the limit \eqref{vs_limit_measure} for any $\psi$ supported in one of the sets $V_r \cap M$ or $W_{j,l} \cap M$. Further, we can choose the sets $V_r$ to be so small that the real part of $d\Theta^{(j)} - d\Theta^{(k)}$ is nonvanishing near $V_r$ if $\gamma(t_j) = \gamma(t_k) = p_r$ but $j \neq k$. This follows since 
$$
\re(d\Theta^{(j)} - d\Theta^{(k)})(p_r) = \dot{\gamma}(t_j)^{\flat} - \dot{\gamma}(t_k)^{\flat} \neq 0.
$$
Here we may need to decrease $\delta$ so that we still have an open cover.

Consider first the case where $\psi \in C_c(W_{j,l} \cap M)$. Here the support of $\psi$ may reach $\partial M$, and we extend $\psi$ by zero outside of $W_{j,l} \cap M$. Suppose that $v_s = e^{is\Theta} a$ where $\Theta = t + \xi_{\alpha} y^{\alpha} + \frac{1}{2}H(t)y \cdot y + O(\abs{y}^3)$ and $a = \tau^{\frac{m-1}{4}}(a_0 + O(\tau^{-1})) \chi(y/\delta')$, and let $\rho = \abs{g}^{1/2}$. Then 
\begin{align*}
 &\int_{M} \abs{v_{\tau+i\lambda}}^2 \psi \,dV_{g} \\
 &= \int_{-\infty}^{\infty} \int_{\mR^{m-1}} e^{-2\lambda t} e^{-\tau \im(H(t))y \cdot y} e^{\tau O(\abs{y}^3)} e^{O(\abs{y})} \tau^{\frac{m-1}{2}} (\abs{a_0}^2 + O(\tau^{-1})) \chi(y/\delta')^2 \psi \rho \,dt \,dy \\
 &= \int_{-\infty}^{\infty} e^{-2\lambda t} \int_{\mR^{m-1}} e^{-\im(H(t))y \cdot y} e^{\tau^{-1/2} O(\abs{y}^3)} e^{\tau^{-1/2} O(\abs{y})} \times \\
 &\qquad (\abs{a_0(t,\tau^{-1/2}y)}^2 + O(\tau^{-1})) \chi(y/\tau^{1/2} \delta')^2 \psi(t,\tau^{-1/2} y) \rho(t,\tau^{-1/2}y) \,dt \,dy.
\end{align*}
Since $\im(H(t))$ is positive definite and $\delta'$ is sufficiently small, the term $e^{-\im(H(t))y \cdot y}$ dominates the other exponentials and one obtains 
\begin{align*}
 &\lim_{\tau \to \infty} \int_{M} \abs{v_{\tau+i\lambda}}^2 \psi \,dV_{g} \\
 &= \int_0^L e^{-2\lambda t} \left( \int_{\mR^{m-1}} e^{-\im(H(t))y \cdot y} \,dy \right) \abs{a_0(t,0)}^2 \psi(t,0) \rho(t,0) \,dt.
\end{align*}
Evaluating the integral over $y$ gives 
$$
\lim_{\tau \to \infty} \int_{M} \abs{v_{\tau+i\lambda}}^2 \psi \,dV_{g} = C_m \int_0^L e^{-2\lambda t} \frac{\abs{a_0(t,0)}^2 \rho(t,0)}{\sqrt{\det\,\im(H(t))}} \psi(t,0) \,dt.
$$
We will prove below that 
\begin{equation} \label{a0_detimht_constant}
\frac{\abs{a_0(t,0)}^2 \rho(t,0)}{\sqrt{\det\,\im(H(t))}} = \text{const}.
\end{equation}
The limit \eqref{vs_limit_measure} will follow upon dividing the family $\{ v_s \}$ by a suitable constant.

If $\psi \in C_c(V_r \cap M)$ (again $\supp(\psi)$ may extend up to $\partial M$), we have 
$$
v_s|_{V_r} = \sum_{\gamma(t_j) = p_r} v^{(j)},
$$
so that on $V_r$ 
$$
\abs{v_s}^2 = \sum_{\gamma(t_j) = p_r} \abs{v^{(j)}}^2 + \sum_{\gamma(t_j) = \gamma(t_k) = p_r, j \neq k} v^{(j)} \overline{v^{(k)}}.
$$
We arranged earlier that $\re(d\Theta^{(j)} - d\Theta^{(k)})$ is nonvanishing near $V_r$ if $\gamma(t_j) = \gamma(t_k) = p_r$ but $j \neq k$. Thus the cross terms give rise to terms of the form 
$$
\int_{V_r \cap M} v^{(j)} \overline{v^{(k)}} \psi \,dV = \int_{V_r \cap M} e^{i\tau \phi} w^{(j)} \overline{w^{(k)}} \psi \,dV
$$
where $\phi = \re(\Theta^{(j)} - \Theta^{(k)})$ has nonvanishing gradient in $V_r$, and $w^{(l)} = e^{is \im(\Theta^{(l)})} e^{-\lambda \re(\Phi^{(l)})} a^{(l)}$. We wish to prove that 
\begin{equation} \label{vr_crossterm_limit}
\lim_{\tau \to \infty} \int_{V_r \cap M} e^{i\tau \phi} w^{(j)} \overline{w^{(k)}} \psi \,dV = 0, \quad j \neq k,
\end{equation}
showing that the cross terms vanish in the limit and the previous computation for $\abs{v^{(l)}}^2$ shows the limit \eqref{vs_limit_measure} also when $\psi$ is supported in some $V_r \cap M$. To show \eqref{vr_crossterm_limit}, let $\eps > 0$, and decompose $\psi = \psi_1 + \psi_2$ where $\psi_1 \in C^{\infty}_c(V_r \cap M)$ and $\norm{\psi_2}_{L^{\infty}(V_r \cap M)} \leq \eps$. Then 
$$
\abs{\int_{V_r \cap M} e^{i\tau \phi} w^{(j)} \overline{w^{(k)}} \psi_2 \,dV} \lesssim \norm{w^{(j)}}_{L^2} \norm{w^{(k)}}_{L^2} \norm{\psi_2}_{L^{\infty}} \lesssim \eps
$$
since $\norm{w^{(l)}}_{L^2} \lesssim \norm{v^{(l)}}_{L^2} \lesssim 1$. For the smooth part $\psi_1$, we employ a non-stationary phase argument and integrate by parts using that 
$$
e^{i\tau \phi} = \frac{1}{i\tau} L(e^{i\tau \phi}), \quad Lw = \langle \abs{d\phi}^{-2} d\phi, dw \rangle.
$$
This gives 
\begin{multline*}
\int_{V_r \cap M} e^{i\tau \phi} w^{(j)} \overline{w^{(k)}} \psi_1 \,dV = \int_{\partial M} \frac{\partial_{\nu} \phi}{i\tau\abs{d\phi}^2} v^{(j)} \overline{v^{(k)}} \psi_1 \,dS \\
+ \frac{1}{i\tau} \int_{V_r \cap M} e^{i\tau \phi} L^t(w^{(j)} \overline{w^{(k)}} \psi_1) \,dV.
\end{multline*}
Since $\norm{v^{(l)}}_{L^2(\partial M)} = O(1)$, the boundary term can be made arbitrarily small as $\tau \to \infty$. As for the last term, the worst behavior is when the transpose $L^t$ acts on $e^{is\im(\Theta^{(l)})}$, and these terms have bounds of the form 
$$
\norm{\abs{d(\im(\Theta^{(j)}))} v^{(j)}}_{L^2} \norm{v^{(k)}}_{L^2} \norm{\psi_1}_{L^{\infty}}.
$$
Here $\abs{d(\im(\Theta^{(j)}))} \lesssim \abs{y}$ if $(t,y)$ are coordinates along the geodesic segment corresponding to $v^{(j)}$, and the computation above for $\norm{v^{(j)}}_{L^2}$ shows that 
$$
\norm{\abs{d(\im(\Theta^{(j)}))} v^{(j)}}_{L^2} \norm{v^{(k)}}_{L^2} \norm{\psi_1}_{L^{\infty}} \lesssim \tau^{-1/2}.
$$
This finishes the proof of \eqref{vr_crossterm_limit} and also of \eqref{vs_limit_measure}.

It remains to show \eqref{a0_detimht_constant}. We have 
$$
\abs{a_0(t,0)}^2 \rho(t,0) = \abs{c_0}^2 e^{-\int_{t_0}^t \re(\eta_0)(s) \,ds} \abs{g(t,0)}^{1/2}.
$$
Note that $\eta_0(t)$ is given by 
\begin{align*}
\eta_0(t) &= \Delta \Theta(t,0) \\
 &= (g^{jk} \partial_{jk} \Theta + \partial_j g^{jk} \partial_k \Theta + \abs{g}^{-1/2} \partial_j(\abs{g}^{1/2}) g^{jk} \partial_k \Theta)(t,0) \\
 &= g^{11} \partial_t^2 \Theta_0 + 2 g^{1 \alpha} \partial_{t \alpha} \Theta_1 + g^{\alpha \beta} \partial_{\alpha \beta} \Theta_2 + \partial_j g^{j1} \partial_t \Theta_0 + \partial_j g^{j \alpha} \partial_{\alpha} \Theta_1 \\
 &\quad \quad + \frac{1}{2} \partial_j(\log \abs{g}) (g^{j1} \partial_t \Theta_0 + g^{j \alpha} \partial_{\alpha} \Theta_1) \\
 &= 2 g^{1 \alpha} \dot{\xi}_{\alpha} + g^{\alpha \beta} H_{\alpha \beta} + (\partial_j g^{jk}) \xi_k + \frac{1}{2} \partial_j(\log \abs{g}) g^{jk} \xi_k.
\end{align*}
The conditions $g^{jk} \xi_k = \delta_1^j$ and $g^{1 \alpha} \dot{\xi}_{\alpha} = g^{1k} \dot{\xi}_k = - (\partial_t g^{1k}) \xi_k$ at $(t,0)$, together with the general fact that $\partial_t(\log \abs{g}) = -g_{jk} \partial_t g^{jk}$, imply that 
$$
\eta_0(t) = g^{1 \alpha} \dot{\xi}_{\alpha} + g^{\alpha \beta} H_{\alpha \beta} + (\partial_{\alpha} g^{\alpha k}) \xi_k + \frac{1}{2} \partial_t(\log \abs{g}).
$$
Recalling the definition of the $B$ and $C$ matrices in \eqref{gaussianbeam_b_c_definition}, this says precisely that 
\begin{align*}
\eta_0(t) &= B_{\alpha}^{\alpha} + C^{\alpha \gamma} H_{\gamma \alpha} + \frac{1}{2} \partial_t(\log \abs{g}) \\
 &= \tr(B(t) + C(t)H(t)) + \frac{1}{2} \partial_t(\log \abs{g}).
\end{align*}
Consequently $\abs{a_0(t,0)}^2 \rho(t,0) = c_0' e^{-\int_{t_0}^t \tr(B(s) + C(s)\re(H(s))) \,ds}$. On the other hand, by \cite[Lemma 2.58]{KKL} solutions of the matrix Riccati equation have the property that 
$$
\det \im(H(t)) = \det \im(H(t_0)) e^{-2 \int_{t_0}^t \tr(B(s) + C(s)\re(H(s))) \,ds}.
$$
This proves the result.
\end{proof}

The proof of Proposition \ref{prop_gaussianbeam_quasimode_reflected} now follows rather quickly from the way we have set up the previous result.

\begin{proof}[Proof of Proposition \ref{prop_gaussianbeam_quasimode_reflected}]
Let $\gamma: [0,L] \to M$ be a nontangential broken ray with endpoints on $E$, and let $0 < t_1 < \ldots < t_N < L$ be the times of reflection. Let $v_s^{(0)}$ be a Gaussian beam quasimode as in Proposition \ref{prop_gaussianbeam_quasimode} associated with the geodesic $\gamma|_{[0,t_1]}$. We will construct another Gaussian beam quasimode $v_s^{(1)}$ associated with $\gamma|_{[t_1,t_2]}$ such that $v_s^{(0)} - v_s^{(1)}|_{\partial M}$ will be small near $\gamma(t_1)$.

In fact, by Proposition \ref{prop_gaussianbeam_quasimode} we have $v_s^{(j)} = e^{is \Theta^{(j)}} a^{(j)}$ near $\gamma(t_1)$, and we can choose the $K$th order jet of $\Theta^{(1)}|_{\partial M}$ at $\gamma(t_1)$ to be equal to that of $\Theta^{(0)}|_{\partial M}$ with the following exception: we always have 
\begin{gather*}
d(\Theta^{(0)})|_{\gamma(t_1)} = \dot{\gamma}(t_1 -)^{\flat}, \\
d(\Theta^{(1)})|_{\gamma(t_1)} = \dot{\gamma}(t_1 +)^{\flat}, \\
\end{gather*}
It follows that 
\begin{gather*}
d(\Theta^{(0)}|_{\partial M})|_{\gamma(t_1)} = \dot{\gamma}(t_1 -)^{\flat}_{tan}, \\
d(\Theta^{(1)}|_{\partial M})|_{\gamma(t_1)} = \dot{\gamma}(t_1 +)^{\flat}_{tan}, \\
\end{gather*}
where we have taken the projections to the cotangent space of $\partial M$ at $\gamma(t_1)$. But by the rule that the angle of incidence equals angle of reflection, it holds that $\dot{\gamma}(t_1 -)^{\flat}_{tan} = \dot{\gamma}(t_1 +)^{\flat}_{tan}$. Thus actually the $K$th order jets of $\Theta^{(0)}|_{\partial M}$ and $\Theta^{(1)}|_{\partial M}$ coincide at $\gamma(t_1)$, and by Proposition \ref{prop_gaussianbeam_quasimode} we can also arrange that the $K$th order jets of $a^{(0)}|_{\partial M}$ and $a^{(1)}|_{\partial M}$ coincide at $\gamma(t_1)$.

Write $f_s = v_s^{(0)} - v_s^{(1)}|_{\partial M}$, and let $(t,y)$ be coordinates near $\gamma(t_1)$ such that $\partial M$ is parametrized by $y \mapsto (t_1,y)$ and $\gamma(t_1)$ corresponds to $(t_1,0)$. Recall that $v_s^{(j)}$ are supported in small tubular neighborhoods of the corresponding geodesic segments. By the above considerations and the construction of $\Theta^{(j)}$ and $a^{(j)}$, and dropping the variable $t_1$ from the notations, the restrictions of $\Theta^{(j)}$ and $a^{(j)}$ to $\partial M$ satisfy 
$$
\Theta^{(j)}(y) = \Theta(y) + \Xi^{(j)}(y), \quad a^{(j)}(y) = a(y)+ b^{(j)}(y)
$$
where $\Theta$ is a polynomial of order $K$, $a = \tau^{\frac{m-1}{4}} \tilde{a} \chi(y/\delta')$ where $\tilde{a}$ is a polynomial of order $K$, and $\abs{\Xi^{(j)}(y)} \leq C \abs{y}^{K+1}$ and $\abs{b^{(j)}(y)} \leq C \tau^{\frac{m-1}{4}} \abs{y}^{K+1} \chi(y/\delta')$ on $\supp(\chi(\,\cdot\,/\delta')$, where $\chi$ is a cutoff function and $\delta'$ is a constant independent of $\tau$ that can be chosen as small as we want (these initially depend on $j$, but since there are finitely many reflections we can choose them independent of $j$). Here $\Theta$ and $\Xi^{(j)}$ are independent of $\tau$, and $a$ and $b^{(j)}$ are mildly $\tau$-dependent and satisfy uniform bounds with respect to $\tau$. Then 
$$
f_s = e^{is\Theta}( (e^{is \Xi^{(0)}} - e^{is \Xi^{(1)}}) a + e^{is \Xi^{(0)}} b^{(0)} - e^{is \Xi^{(1)}} b^{(1)} ).
$$
We have 
$$
e^{is \Xi^{(0)}} - e^{is \Xi^{(1)}} = is(\Xi^{(0)} - \Xi^{(1)}) \int_0^1 e^{is(r \Xi^{(0)} + (1-r) \Xi^{(1)})} \,dr
$$
and consequently near $y = 0$ 
$$
\abs{e^{is \Xi^{(0)}} - e^{is \Xi^{(1)}}} \leq C\tau \abs{y}^{K+1} e^{C \tau \abs{y}^{K+1}}.
$$
Thus, near $y = 0$â 
$$
\abs{f_s(y)} \leq C \tau^{\frac{m-1}{4}} e^{-\tau \im(\Theta)} \tau \abs{y}^{K+1} e^{C \tau \abs{y}^{K+1}} \chi(y/\delta').
$$
Using that the Hessian of $\im(\Theta)$ at $0$ is positive definite and choosing $\delta'$ sufficiently small, we have 
$$
\abs{f_s(y)} \leq C \tau^{\frac{m-1}{4}} e^{-c \tau \abs{y}^2} \tau \abs{y}^{K+1} \chi(y/\delta').
$$
Integrating the square of $\abs{f_s}$ over $\mR^{m-1}$ and changing $y$ to $\tau^{-1/2} y$, we obtain 
$$
\norm{f_s}_{L^2(R_1)} = O(\tau^{-\frac{K-1}{2}})
$$
where $R_1$ is a small neighborhood of $\gamma(t_1)$ on $\partial M$ containing the set of interest.

Repeating this construction for the other points of reflection, we end up with a quasimode 
$$
v_s = \sum_{j=0}^N (-1)^j v_s^{(j)}
$$
that is supported in a small neighborhood of the broken ray $\gamma$. Since all points of reflection are distinct, we can arrange that the quasimode satisfies 
$$
\norm{v_s|_R}_{L^2(R)} = O(\tau^{-\frac{K-1}{2}}).
$$
It also satisfies 
$$
\norm{(-\Delta - s^2) v_s}_{L^2(M)} = O(\tau^{-K}), \quad \norm{v_s}_{L^2(M)} = O(1).
$$
Replacing $K$ by $2K+1$, we have proved all the other statements in the proposition except for the expression of the limit measure. To do this, we consider the finitely many points where the full broken ray $\gamma$ self-intersects or reflects, and decompose the terms $v_s^{(j)}$ as in the proof of Proposition \ref{prop_gaussianbeam_quasimode} to parts living in small neighborhoods of the self-intersection and reflection points and parts away from these points. Now all self-intersection points are in the interior or on $E$ and all self-intersections must be transversal, and also all reflections are transversal. Consequently, when forming $\abs{v_s}^2$, the cross terms arising from different parts living near the same self-intersection or reflection point contribute an $o(1)$ term by non-stationary phase as in the proof of Proposition \ref{prop_gaussianbeam_quasimode}. Thus the limit measure of $\abs{v_s}^2 \,dV_g$ is indeed the measure $e^{-2\lambda t} \delta_{\gamma}$, where $\delta_{\gamma}$ is the delta function of the broken ray $\gamma$.
\end{proof}

\section{Recovering the broken ray transform} \label{sec_brokenraytransform}

In this section we give the proof of Theorem \ref{theorem_main_brokenrays} concerning the recovery of integrals over broken rays.

\begin{proof}[Proof of Theorem \ref{theorem_main_brokenrays}]
The proof is very similar to the proof of Theorem \ref{theorem_main_local_integrals}, except that we use reflected Gaussian beam quasimodes instead of WKB type quasimodes. Let $\gamma: [0,L] \to M_0$ be a nontangential broken ray with endpoints on $E$, and let $\lambda > 0$. Let also $\tilde{g} = e \oplus g_0$ and $\tilde{q}_j = c(q_j - c^{\frac{n-2}{4}} \Delta_g(c^{-\frac{n-2}{4}}))$. Consider the complex frequency 
$$
s = \tau + i\lambda
$$
where $\tau > 0$ will be large. We look for solutions 
\begin{align*}
\tilde{u}_1 &= e^{-s x_1}(v_s(x') + r_1), \\
\tilde{u}_2 &= e^{s x_1}(v_s(x') + r_2),
\end{align*}
of the equations $(-\Delta_{\tilde{g}} + \tilde{q}_1) \tilde{u}_1 = 0$, $(-\Delta_{\tilde{g}} + \overline{\tilde{q}_2}) \tilde{u}_2 = 0$ in $M$. Here $v_s \in C^{\infty}(M_0)$ is the quasimode constructed in Proposition \ref{prop_gaussianbeam_quasimode_reflected} that concentrates near the given broken ray $\gamma$ and is small on $\partial M_0 \setminus E$.

Since $\Delta_{\tilde{g}} = \partial_1^2 + \Delta_{g_0}$, the function $\tilde{u}_1$ is a solution if and only if 
$$
e^{s x_1} (-\Delta_{\tilde{g}} + \tilde{q}_1) (e^{-s x_1} r_1) = -(-\Delta_{g_0} + \tilde{q}_1 - s^2) v_s(x') \quad \text{in } M.
$$
We look for a solution in the form $r_1 = e^{i\lambda x_1} r_1'$ where $r_1'$ satisfies 
$$
e^{\tau x_1} (-\Delta_{\tilde{g}} + \tilde{q}_1) (e^{-\tau x_1} r_1') = f \quad \text{in } M
$$
with 
$$
f = -e^{-i\lambda x_1}(-\Delta_{g_0} + \tilde{q}_1 - s^2) v_s(x').
$$
To arrange that $\tilde{u}_1|_{\Gamma_i} = 0$, fix some small $\delta > 0$, let $S_{\pm}$ and $S_0$ be the sets in Proposition \ref{prop_carleman_solvability} with Carleman weight $\varphi(x) = -x_1$, and consider the boundary condition 
$$
e^{\tau \phi} r_1'|_{S_- \cup S_0} = e^{\tau \phi} f_-
$$
where 
$$
f_- = \left\{ \begin{array}{cl} -e^{-i\lambda x_1} v_s(x'), & \quad \text{on } \Gamma_i, \\ 0, & \quad \text{on } (S_- \cup S_0) \setminus \Gamma_i. \end{array} \right.
$$

For any fixed $K > 0$, by Proposition \ref{prop_gaussianbeam_quasimode_reflected} and by the condition that $\Gamma_i \subset \mR \times (\partial M_0 \setminus E)$ we may assume that the following bounds are valid:
$$
\norm{f}_{L^2(M)} = O(1), \quad \norm{f_-}_{L^2(S_-)} = 0, \quad \norm{f_-}_{L^2(S_0)} = O(\tau^{-K}).
$$
It follows from Proposition \ref{prop_carleman_solvability} that there is a solution $r_1'$ satisfying the above boundary condition and having the estimate 
$$
\norm{r_1'}_{L^2(M)} = O(\tau^{-1}).
$$

Choosing $r_1'$ as described above and choosing $r_1 = e^{i\lambda x_1} r_1'$, we have produced a solution $\tilde{u}_1 \in H_{\Delta_{\tilde{g}}}(M)$ of the equation $(-\Delta_{\tilde{g}} + \tilde{q}_1) \tilde{u}_1 = 0$ in $M$, having the form 
$$
\tilde{u}_1 = e^{-s x_1}(v_s(x') + r_1)
$$
and satisfying 
$$
\supp(\tilde{u}_1|_{\partial M}) \subset \overline{\partial M_+ \cup \partial M_- \cup \Gamma_a}
$$
and $\norm{r_1}_{L^2(M)} = O(\tau^{-1})$ as $\tau \to \infty$. Repeating this construction for the Carleman weight $\phi(x) = x_1$, we obtain a solution $\tilde{u}_2 \in H_{\Delta_{\tilde{g}}}(M)$ of the equation $(-\Delta_{\tilde{g}} + \overline{\tilde{q}_2}) \tilde{u}_2 = 0$ in $M$, having the form 
$$
\tilde{u}_2 = e^{s x_1}(v_s(x') + r_2)
$$
satisfying the same support condition and bound for $\norm{r_2}_{L^2(M)}$.

Writing $u_j = c^{-\frac{n-2}{4}} \tilde{u}_j$, Lemma \ref{lemma_conformal_scaling_solutions} shows that $u_j \in H_{\Delta_g}(M)$ are solutions of $(-\Delta_g + q_1) u_1 = 0$ and $(-\Delta_g + \overline{q_2}) u_2 = 0$ in $M$. Then Proposition \ref{prop_integralidentity} implies that 
$$
\int_M (q_1-q_2) u_1 \overline{u_2} \,dV_g = 0.
$$
We extend $q_1-q_2$ by zero to $\mR \times M_0$. Inserting the expressions for $u_j$, and using that $dV_g = c^{n/2} \,dx_1 \,dV_{g_0}(x')$, we obtain 
$$
\int_{M_0} \int_{-\infty}^{\infty} (q_1-q_2) c e^{-2i\lambda x_1}(\abs{v_s(x')}^2 + v_s \overline{r_2} + \overline{v_s} r_1 + r_1 \overline{r_2}) \,dx_1 \,dV_{g_0}(x') = 0.
$$
Since $\norm{r_j}_{L^2(M)} = O(\tau^{-1})$ as $\tau \to \infty$, Proposition \ref{prop_gaussianbeam_quasimode_reflected} implies that 
$$
\int_0^L e^{-2\lambda t}(c(q_1-q_2))\ehat(2\lambda,\gamma(t)) \,dt = 0.
$$
This concludes the proof.
\end{proof}

\bibliographystyle{alpha}

\end{document}